\documentclass[12pt,reqno]{amsart}
\usepackage[a4paper, hmargin={2.7cm,2.7cm},vmargin={3.3cm,3.3cm}]{geometry}
\usepackage{latexsym,amsmath,amssymb,amscd}
\usepackage[noadjust]{cite}
\usepackage{amsfonts}
\usepackage{amsthm}
\usepackage{enumerate}
\usepackage{amsthm}
\usepackage{todonotes}
\usepackage{cleveref}
\usepackage[all]{xy}
\usepackage{dsfont}
\usepackage{etoolbox}
\usepackage{verbatim}

\numberwithin{equation}{section}

\makeatletter
\patchcmd{\ttlh@hang}{\parindent\z@}{\parindent\z@\leavevmode}{}{}
\patchcmd{\ttlh@hang}{\noindent}{}{}{}
\makeatother

\newtheorem{theorem}{Theorem}[section]
\newtheorem{lemma}[theorem]{Lemma}
\newtheorem{proposition}[theorem]{Proposition}
\newtheorem{corollary}[theorem]{Corollary}

\theoremstyle{definition}
\newtheorem{definition}[theorem]{Definition}

\theoremstyle{remark}
\newtheorem{remark}[theorem]{Remark}

\DeclareMathOperator{\PW}{PW}
\DeclareMathOperator*{\tr}{tr}
\DeclareMathOperator*{\supp}{supp}

\newcommand{\R}{\mathbb{R}}

\newcommand{\N}{\mathbb{N}}

\newcommand{\Z}{\mathbb{Z}}

\newcommand{\T}{\mathbb{T}}

\newcommand{\ve}{\varepsilon}

\newcommand{\lan}{\langle}
\newcommand{\ran}{\rangle}

\title[On exponential frames near the critical density]{On exponential frames near the critical density}

\author{Marcin Bownik}
\address{Department of Mathematics, University of Oregon, Eugene, OR 97403--1222, USA}
\email{mbownik@uoregon.edu}

\author{Jordy Timo van Velthoven}
\address{Faculty of Mathematics,
University of Vienna, 
Oskar-Morgenstern-Platz 1,
1090 Vienna, Austria}
\email{jordy-timo.van-velthoven@univie.ac.at}

\makeatletter
\@namedef{subjclassname@2020}{\textup{2020} Mathematics Subject Classification}
\makeatother



\begin{document}

\begin{abstract}
Given a compact set $\Omega \subseteq \R$ of Lebesgue measure $|\Omega|$ and $\varepsilon > 0$, we show the existence of a set $\Lambda \subseteq \R$ of uniform density $D (\Lambda) \leq (1+\varepsilon) |\Omega|$ such that the exponential system $\{ \exp(2\pi i \lambda \cdot) \mathds{1}_{\Omega}: \lambda \in \Lambda \}$ is a frame for $L^2 (\Omega)$ with frame bounds $A |\Omega|, B |\Omega|$ for constants $A,B$ only depending on $\varepsilon$. This solves a problem on the frame bounds of an exponential frame near the critical density posed by Nitzan, Olevskii and Ulanovskii.
 We also prove an extension to locally compact abelian groups, which improves a result by Agora, Antezana and Cabrelli by providing frame bounds involving the spectrum.
\end{abstract}

\maketitle

\section{Introduction}
Let $\Omega \subseteq \R$ be a compact set. Given a discrete set $\Lambda \subseteq \R$, consider the associated exponential functions 
\[
 e_{\lambda} (t) = e^{2 \pi i \lambda t}, \quad t \in \mathbb{R}, \; \lambda \in \Lambda.
\]
The system $\{e_{\lambda} \mathds{1}_{\Omega}\}_{\lambda \in \Lambda}$ is said to be a \emph{frame} for $L^2 (\Omega)$ if there exist constant $A, B>0$, called \emph{frame bounds}, satisfying
\[
 A \| f \|^2 \leq \sum_{\lambda \in \Lambda} |\langle f, e_{\lambda} \rangle |^2 \leq B \| f \|^2 \quad \text{for all} \quad f \in L^2 (\Omega).
\]
For example, an orthogonal basis $\{e_{\lambda} \mathds{1}_{\Omega}\}_{\lambda \in \Lambda}$ is a frame with bounds $A=B=|\Omega|$, where $|\Omega|$ denotes the Lebesgue measure of $|\Omega|$. 
 Landau's necessary density conditions \cite{landau1967necessary} assert that if $\{e_{\lambda} \mathds{1}_{\Omega}\}_{\lambda \in \Lambda}$ is a frame for $L^2 (\Omega)$, then
 \[
  D^- (\Lambda) := \liminf_{r \to \infty} \inf_{x \in \mathbb{R}} \frac{\# (\Lambda \cap [x, x+r])}{r} \geq |\Omega|.
 \]
In addition, if $\{e_{\lambda} \mathds{1}_{\Omega}\}_{\lambda \in \Lambda}$ is a Riesz sequence in $L^2 (\Omega)$, then also
 \[
  D^+ (\Lambda) := \limsup_{r \to \infty} \sup_{x \in \mathbb{R}} \frac{\# (\Lambda \cap [x, x+r])}{r} \leq |\Omega|.
 \]
In particular, any Riesz basis $\{e_{\lambda} \mathds{1}_{\Omega}\}_{\lambda \in \Lambda}$ for $L^2 (\Omega)$ has density  $D^- (\Lambda) = D^+ (\Lambda) = |\Omega|$. In general, a discrete set $\Lambda \subseteq \mathbb{R}$ is said to be of \emph{uniform density} if $D^- (\Lambda) = D^+ (\Lambda)$ in which case we denote its density simply by $D(\Lambda)$.

Exponential frames for $L^2(\Omega)$ are easily shown to exist. For example, the set $\Lambda = \mathbb{Z}$ defines an orthonormal basis $\{e_{\lambda} \mathds{1}_{[0,1]} \}_{\lambda \in \Lambda}$ for $L^2 ([0, 1])$,
and hence defines a frame $\{e_{\lambda} \mathds{1}_{\Omega}\}_{\lambda \in \Lambda}$ for $L^2 (\Omega)$ for any set $\Omega \subseteq [0, 1]$. However, the construction of exponential frames for $L^2(\Omega)$ with density close to the critical value $|\Omega|$ or with frame bounds comparable to $|\Omega|$  is remarkably nontrivial and has been considered, among others, in \cite{nitzan2013few, nitzan2016exponential, agora2015multi, marzo2006riesz}. More precisely, the existence of frames $\{e_{\lambda} \mathds{1}_{\Omega}\}_{\lambda \in \Lambda}$ with density $D^- (\Lambda) \leq (1+\ve) |\Omega|$ for a given $\ve > 0$ has been shown in \cite{agora2015multi, nitzan2013few, marzo2006riesz}, whereas the existence of frames $\{e_{\lambda} \mathds{1}_{\Omega}\}_{\lambda \in \Lambda}$ with frame bounds $c |\Omega|, C|\Omega|$ for absolute constants $c,C>0$ has been shown in \cite{nitzan2016exponential}. The question whether there exist frames possessing both of the aforementioned properties was posed by Nitzan, Olevskii and Ulanovskii  \cite{nitzan2013few}. Explicitly, it is asked\footnote{The problem \cite[Section 6.2, Open problem 1]{nitzan2013few} concerns the equivalent problem of sampling in Paley-Wiener spaces; see Corollary \ref{cor:paley}.} in \cite[Section 6.2, Open problem 1]{nitzan2013few} whether \cite[Theorem 1]{nitzan2013few} and \cite[Theorem 3]{nitzan2013few} can be combined, meaning that, given $\varepsilon > 0$, there exists $\Lambda \subseteq \R$ with density 
\[
D^- (\Lambda) \leq (1+\ve)|\Omega|
\] 
and such that 
\[
A(\varepsilon) |\Omega|  \| f \|^2  \leq \sum_{\lambda \in \Lambda} |\langle f, e_{\lambda}  \rangle |^2 \leq B(\varepsilon) |\Omega|  \| f \|^2 \quad \text{for all} \quad f \in L^2 (\Omega),
\]
is fulfilled with absolute constants $A(\ve)$ and $B(\ve)$ depending only on $\ve$. 
 The aim of this paper is to answer this question. Our first main result is as follows.

\begin{theorem} \label{thm:intro}
Let $\varepsilon > 0$. Given a compact set $\Omega \subseteq \R$,  there exists $\Lambda \subseteq \mathbb{R}$ of uniform density $D(\Lambda) \leq (1+\varepsilon) |\Omega|$ such that
\begin{align} \label{eq:int}
 A(\varepsilon) |\Omega| \| f \|^2  \leq \sum_{\lambda \in \Lambda} |\langle f, e_{\lambda} \rangle |^2 \leq B(\varepsilon) |\Omega| \| f \|^2  \quad \text{for all} \quad f \in L^2 (\Omega)
\end{align}
for some positive constants $A(\varepsilon), B(\varepsilon)$  depending only on $\varepsilon$. Moreover, if $\Omega$ is contained in an interval of length $d>0$, then $\Lambda$ can be chosen to satisfy $\Lambda \subseteq d^{-1} \mathbb{Z}$. 
\end{theorem}

Our proof of \Cref{thm:intro} is modeled on that of Nitzan, Olevskii, and Ulanovskii \cite[Theorem 1]{nitzan2013few} who showed a lower bound in \eqref{eq:int}. The question whether it is possible to obtain both lower and upper bounds in \eqref{eq:int} was posed as \cite[Section 6.2, Open problem 1]{nitzan2013few}. For proving the upper bound in \Cref{thm:intro}, we revisit a sparsification result for finite-dimensional frames due to Batson, Spielman, and Srivastava \cite{batson2014twice}. Their result  \cite[Theorem 3.1]{batson2014twice} allows to remove elements from an overcomplete finite-dimensional Parseval frame for $\mathbb{C}^n$ while leaving a weighted frame with cardinality arbitrary close to  $n$. We mention that the result of Batson, Spielman, and Srivastava was already used in \cite{nitzan2013few}. Using a result on scalable frames by the first named author \cite{bownik2024selector}, 
we obtain a version of the sparsification result \cite[Theorem 3.1]{batson2014twice} that additionally provides a lower bound on the nonzero weights; see \Cref{thm:bss2}. In turn, this allows us to remove elements from an overcomplete Parseval frame for $\mathbb{C}^n$ and obtain an \emph{unweighted} frame for $\mathbb{C}^n$ with explicit frame bounds and cardinality arbitrary close to $n$; see Lemma \ref{lem:batson}. The latter result is the key ingredient that allows us to obtain the upper bound in \eqref{eq:int} based on the methods used in \cite{nitzan2013few}. 
We mention that using weak limit techniques, Theorem \ref{thm:intro} can be shown to imply the existence of exponential frames on unbounded spectra with frame bounds satisfying \eqref{eq:int}, see  \cite{nitzan2016exponential}. However, weak limit techniques do not imply that a resulting frame is near the critical density. 

In addition, we provide an extension of \Cref{thm:intro} to exponential frames on locally compact abelian (LCA) groups. Such frames have been studied in various papers, see \cite{grochenig2008landau, agora2015multi, richard2020on, enstad2024dynamical, bownik2021multiplication, agora2019existence, antezana2022universal}. To be explicit, let $G$ be a second countable LCA group with Haar measure $\mu_G$ and the dual group $\widehat{G}$ with dual measure $\mu_{\widehat{G}}$. Given $t \in G$, we define the function $e_t \in \widehat{\widehat{G}}$ by $e_t (\chi) = \chi(t)$ for $\chi \in \widehat{G}$. 
Gr\"ochenig, Kutyniok, and Seip \cite{grochenig2008landau} extended Landau's result to LCA groups, see also \cite{richard2020on, enstad2024dynamical}. If $\Omega \subseteq \widehat{G}$ is a compact set and  $\{e_t \mathds{1}_{\Omega} \}_{t \in T}$ 
is a frame for $L^2 (\Omega)$, then
\[
D^- (T) := \liminf_{n \to \infty} \inf_{x \in G} \frac{\# (T \cap (x+K_n))}{\mu_G (K_n)} \geq \mu_{\widehat{G}} (\Omega), 
\]
where $(K_n)_{n \in \mathbb{N}}$ denotes a sequence of compact sets $K_n \subseteq G$ forming a \emph{strong F\o lner sequence}; see  Section \ref{sec:beurling} for more details on such sequences and associated Beurling densities. Similarly, if the system $\{e_t \mathds{1}_{\Omega} \}_{t \in T}$ 
is a Riesz sequence in $L^2 (\Omega)$, then
\[
D^+ (T) := \limsup_{n \to \infty} \sup_{x \in G} \frac{\# (T \cap (x+K_n))}{\mu_G (K_n)} \leq \mu_{\widehat{G}} (\Omega).
\]
In particular, if $\{e_t \mathds{1}_{\Omega} \}_{t \in T}$ is a Riesz basis for $L^2 (\Omega)$, then $T \subseteq G$ has \emph{uniform Beurling density} if $D(T) := D^- (T) = D^+ (T)$ is equal to the critical density $\mu_{\widehat{G}} (\Omega)$.

Agora, Antezana, and Cabrelli \cite{agora2015multi} showed that there exist exponential frames and Riesz sequences with density arbitrary close to the critical density. Although the papers \cite{grochenig2008landau, agora2015multi} use other notion of density defined by a reference lattice, their results can be rephrased in terms of Beurling density associated to F\o lner sequences as above.
 Our main result improves the existence of such exponential frames by providing frame bounds involving the spectrum. The precise statement is as follows. 

\begin{theorem} \label{thm:intro_LCA}
Let $G$ be a second countable LCA group with Haar measure $\mu_G$. Let $\widehat{G}$ be its dual group with dual measure $\mu_{\widehat{G}}$. 

Let $\varepsilon > 0$. For any compact set $\Omega \subseteq \widehat{G}$, there exists $T \subseteq G$ of uniform Beurling density $D(T) \leq (1+\varepsilon) \mu_{\widehat{G}} (\Omega)$ such that
\begin{align} \label{eq:inequalities2}
 A(\varepsilon) \mu_{\widehat G}(\Omega)  \| f \|^2 \leq \sum_{t \in T} |\langle f, e_t \rangle |^2 \leq B(\varepsilon) \mu_{\widehat G} (\Omega) \| f \|^2 \qquad \text{for all } f \in L^2 (\Omega; \mu_{\widehat G})
\end{align}
for some positive constants $A(\varepsilon), B(\varepsilon)$  depending only on $\varepsilon$. 
\end{theorem}

\Cref{thm:intro_LCA} is shown by reducing to special cases of LCA groups and lifting each of these versions to a more general class of groups. More precisely, we first prove a version of \Cref{thm:intro_LCA} for so-called \emph{elemental groups}, which are groups of the form $\mathbb{R}^d \times \mathbb{T}^n \times \mathbb{Z}^{\ell} \times F$, where $d, n$ and $\ell$ are nonnegative integers and $F$ is a finite abelian group. The proof of \Cref{thm:intro_LCA} for elemental groups is based on the proof of the real line. Second, we prove a version of \Cref{thm:intro_LCA} under the assumption that the dual group $\widehat{G}$ is compactly generated. For such a group $\widehat{G}$, one can pass to an elemental quotient group $\widehat{G} / K$, where $K$ is compact subgroup. The lifting of a frame from $\widehat{G}/K$ to a compactly generated group $\widehat{G}$ is inspired by the approach in \cite{agora2015multi} and hinges on the use of quasi-dyadic cubes, see \Cref{sec:quasidyadic}. Lastly, in order to prove \Cref{thm:intro_LCA} in full generality, we use a result from \cite{grochenig2008landau} that allows us to pass from the dual group $\widehat{G}$ to the (compactly generated) group generated by the spectrum $\Omega \subseteq \widehat{G}$. 

The organization of the paper is as follows.  Section \ref{sec:revisiting} is devoted to revisiting the sparsification result of Batson, Spielman and Srivastava for finite-dimensional frames.  A direct proof of \Cref{thm:intro} is given in Section \ref{sec:realline}.  The proof of \Cref{thm:intro_LCA} is carried out in Section \ref{sec:LCA}.

\subsection*{Notation} The set of natural numbers is denoted by $\mathbb{N} = \{ 1, 2, ...\}$ and we write $\mathbb{N}_0 = \mathbb{N} \cup \{0\}$. For a column vector $v \in \mathbb{C}^n$, we write $v^*$ for its conjugate transpose. The identity operator on $\mathbb{C}^n$ is denoted by $\mathbf{I}_n$. The indicator function of a set $\Omega$ is written as $\mathds{1}_{\Omega}$. 

\section{Revisiting Batson, Spielman, and Srivastava's theorem} \label{sec:revisiting}

We start with the sparsification theorem due to Batson, Spielman and Srivastava \cite[Theorem 3.1]{batson2014twice}, see also the exposition \cite[Theorem 2.1]{naor2012sparse}.

\begin{theorem}[\cite{batson2014twice}] \label{thm:bss}
 Let $\{v_i\}_{i = 1}^m$ be vectors in $\mathbb C^n$ satisfying
\begin{equation}\label{eq:bss1}
\sum_{i=1}^m v_i v_i^* = \mathbf{I}_n.
\end{equation}
For any $d>1$, there exist scalars $s_i \ge 0$, $i = 1, ..., m$, such that $\#\{ i : s_i \ne 0 \} \le \lceil dn \rceil$ and
\begin{align} \label{eq:bss2}
\bigg(1 - \frac{1}{\sqrt{d}} \bigg)^2 \mathbf{I}_n  \le \sum_{i=1}^m s_i v_i v_i^* \le
\bigg(1 + \frac{1}{\sqrt{d}} \bigg)^2 \mathbf{I}_n.
\end{align}
\end{theorem}

\begin{remark}
 Strictly speaking, the proofs given in \cite{batson2014twice,naor2012sparse} consider only real vector spaces, but the arguments also work for complex spaces. In addition, the proofs in \cite{batson2014twice,naor2012sparse}  appear to yield an additional term in the upper bound \eqref{eq:bss2} in case $dn$ is not an integer.
However, using the decreasing function
\[
(1,\infty) \ni d\mapsto  \bigg( 1+ \frac1{\sqrt{d}}\bigg)^2   \bigg( 1- \frac1{\sqrt{d}}\bigg)^{-2},
\]
it follows that, if $dn \not \in \N$, then choosing $d'>d$ such that $ d'n =  \lceil dn \rceil$ and applying Theorem \ref{thm:bss} for $d'>1$, yields the conclusion of Theorem \ref{thm:bss} after a simple rescaling of coefficients $s_i$.
\end{remark}

The proof of Theorem \ref{thm:bss} given in \cite{batson2014twice, naor2012sparse} does not appear to provide any additional control over the scalars $s_i \geq 0$ appearing in Equation \eqref{eq:bss2}. It is our aim to provide such control by exploiting the following theorem, which is a finite variant of the result \cite[Theorem 7.1]{bownik2024selector} on scalable frames.

\begin{theorem}[\cite{bownik2024selector}] \label{thm:scal}
Let $I$ be a finite set and $\delta>0$. Suppose that $\{T_i\}_{i\in I}$ is a family of positive trace class operators in a separable Hilbert space $\mathcal{H}$ satisfying
\begin{equation*}\label{scal1}
\tr(T_i)\le \delta \qquad\text{for all }i \in I.
\end{equation*}
Let $\{a_i\}_{i\in I}$ be a sequence of positive numbers and define
\begin{equation*}\label{scal0}
T:= \sum_{i\in I} a_i T_i.
\end{equation*}
For any $0<\ve<1$, there exists a finite set $I'$ and a sampling function $\pi:I' \to I$ such that
\begin{equation*}\label{scal2}
\bigg\| \frac 1a \sum_{n\in I'} T_{\pi(n)} -  T \bigg\| < \ve,
\end{equation*}
for some constant $a > 0$ satisfying
$c_0 \frac{\delta}{\ve^2} \le a \le  2 c_0 \frac{\delta}{\ve^2}$, where $c_0>0$ is an absolute constant.
\end{theorem}

\begin{proof} The claim follows from \cite[Theorem 7.1]{bownik2024selector} by
appending an infinite collection of zero operators to the family $\{T_i\}_{i\in I}$. Indeed, the sampling function $\pi$ can only repeat the same value $i$ finitely many times when $T_i \ne 0$.
\end{proof}

Combining Theorems \ref{thm:bss} and  \ref{thm:scal}, we obtain an improvement of the theorem of Batson, Spielman, and Srivastava, which provides control over the weights in \eqref{eq:bss2}. The precise statement is as follows.

\begin{theorem}\label{thm:bss2}
 There exists an absolute constant $c_1 > 0$ with the following property:  If $\{v_i\}_{i = 1}^m$ is a family of nonzero vectors in $\mathbb{C}^n$ satisfying
 \[
 \sum_{i = 1}^m v_i v_i^* = \mathbf{I}_n,
 \]
  then for any $d>1$ there exist scalars $s_i \ge 0$, $i = 1, ..., m$, such that:
\begin{enumerate}
\item $\#\{ i : s_i \ne 0 \} \le \lceil dn \rceil$;
\item each $s_i$ is an integer multiple of $b$, where 
\[ c_1 \min_{i=1,...,m} ||v_i||^{-2} (1- 1/\sqrt{d})^4 \leq b \leq  2c_1  \min_{i = 1, ..., m} ||v_i||^{-2} (1- 1/\sqrt{d})^4; \]
\item 
\begin{align*}
\frac{1}{2} \bigg(1 - \frac{1}{\sqrt{d}} \bigg)^2 \mathbf{I}_n \le \sum_{i = 1}^m s_i v_i v_i^* \le 2 \bigg(1 + \frac{1}{\sqrt{d}} \bigg)^2 \mathbf{I}_n .
\end{align*}
\end{enumerate}
\end{theorem}
\begin{proof}
Let $I = \{ 1, ..., m \}$. By Theorem \ref{thm:bss}, there exists a family $\{ s'_i \}_{i \in I}$ of scalars $s'_i \geq 0$ such that $\# \{ i \in I : s'_i \neq 0 \} \leq \lceil dn \rceil$ and
\[
 \bigg(1 - \frac{1}{\sqrt{d}} \bigg)^2  \mathbf{I}_n \leq \sum_{i = 1}^m s'_i v_i v_i^* \leq \bigg(1 + \frac{1}{\sqrt{d}} \bigg)^2  \mathbf{I}_n.
\]
An application of Theorem \ref{thm:scal} with $a_i = s'_i$, $\delta = \max_{i \in I} \| v_i \|^2$, $T_i = v_i v_i^*$ and $\ve = \frac{1}{2} (1- 1/\sqrt{d})^2$ yields a finite set $I'$ and a sampling function $\pi : I' \to I$ such that
\[
\bigg\| \frac{1}{a} \sum_{n \in I'} T_{\pi(n)} - \sum_{i \in I} s_i' T_i \bigg\| < \frac{1}{2} \bigg(1 - \frac{1}{\sqrt{d}} \bigg)^2
\]
for a constant $a > 0$ satisfying
\[
4 c_0  \max_{i \in I} \| v_i \|^2 / (1- 1/\sqrt{d})^4 \leq a \leq 8 c_0  \max_{i \in I} \| v_i \|^2 / (1- 1/\sqrt{d})^4,
\]  
where $c_0 > 0$ is the absolute constant from Theorem \ref{thm:scal}. 
Therefore,
\[
 \frac{1}{2} \bigg(1 - \frac{1}{\sqrt{d}} \bigg)^2 \mathbf{I}_n \leq \frac{1}{a} \sum_{n \in I'} v_{\pi(n)} v^*_{\pi(n)} \leq 2 \bigg(1 + \frac{1}{\sqrt{d}} \bigg)^2 \mathbf{I_n}.
\]
Since $I'$ is a finite set, there exist integers $l_i \in \mathbb{N}_0$ such that
\[
 \sum_{n \in I'} v_{\pi(n)} v^*_{\pi(n)} = \sum_{i \in I} l_i v_i v_i^*.
\]
Setting $s_i := l_i / a$ for $i \in I$ yields the conclusions (ii) and (iii) with $b=1/a$ and $c_1=1/(8c_0)$.
\end{proof}

There are two results in the literature that are closely related to \Cref{thm:bss2}. First, with notation as in \Cref{thm:bss2}, Balan, Casazza and Landau \cite[Lemma 3.2]{balan2011redundancy} proved a sparsification result for finite-dimensional frames such that the cardinality $\# \{ i : s_i \neq 0 \}$ can be chosen arbitrary close to $n$ and such that $s_i = 1$ whenever $s_i \neq 0$. However, the frame bounds in \cite[Lemma 3.2]{balan2011redundancy} depend in an implicit manner on $d$ and $m/n$. Second,
a sparsification theorem with equal weights and absolute frame bounds has been proven by Friedland and Youssef \cite{friedland2019approximating}. More precisely, their result \cite[Corollary 3.1]{friedland2019approximating} provides a sparsification with control over the weights $s_i$ in case $v_i$ have equal norms. However, in contrast to Theorem \ref{thm:bss2}, it appears that \cite[Corollary 3.1]{friedland2019approximating} does not allow the cardinality of $ \{ i : s_i \neq 0 \}$ to be arbitrary close to the ambient dimension $n$. The significance of Theorem \ref{thm:bss2} for the present paper is that it simultaneously allows the cardinality $\# \{ i : s_i \neq 0 \}$ to be arbitrarily close to $n$, provides estimates for the nonzero weights, and yields  frame bounds only depending on $d$.

Using  Theorem \ref{thm:bss2}, we prove the following result, which is the key ingredient that we will actually use in the present paper. Lemma \ref{lem:batson} provides the upper bound missing in \cite[Corollary 1]{nitzan2013few}.

\begin{lemma} \label{lem:batson}
There exist absolute constants $c, C>0$ with the following property: If $\{v_i\}_{i = 1}^m$ is a family of vectors $v_i \in \mathbb{C}^n$ with norms $\| v_i \|^2 = n/m$ satisfying
 \begin{align} \label{eq:parseval}
 \sum_{i = 1}^m |\langle v, v_i \rangle |^2 = \| v \|^2 \qquad \text{for all }  v \in \mathbb{C}^n,
 \end{align}
 then for every $\varepsilon > 0$, there exists a subset $J \subseteq \{ 1, ..., m\}$ of cardinality $\#J \leq \lceil (1+\varepsilon) n \rceil$ such that
\begin{align} \label{eq:batson}
  c (1-1/\sqrt{1+\ve})^2 \frac{n}{m} \| v \|^2 \leq \sum_{i \in J} |\langle v, v_i \rangle |^2 \leq C \frac{1}{(1-1/\sqrt{1+\ve})^4} \frac{n}{m} \| v \|^2
\end{align}
for all $v \in \mathbb{C}^n$. 
\end{lemma}
\begin{proof}
Note that the assumption \eqref{eq:parseval} implies that $\sum_{i = 1}^m v_i v_i^* = \mathbf{I}_n$. 
Hence, an application of Theorem \ref{thm:bss2} with $d := 1+\varepsilon$ yields a family $\{s_i \}_{i \in J}$ of scalars $s_i > 0$ with $\# J \leq \lceil (1+\varepsilon) n\rceil$
such that
\begin{align} \label{eq:batson2}
 \frac{1}{2} \bigg(1 - \frac{1}{\sqrt{1+\ve}} \bigg)^2 \| v \|^2 \leq \sum_{i \in J} s_i |\langle v, v_i \rangle |^2 \leq 2 \bigg(1 + \frac{1}{\sqrt{1+\ve}} \bigg)^2\| v \|^2 
\end{align}
for all $v \in \mathbb{C}^n$. 
To simplify notation, we set $c(\ve) = (1-1/\sqrt{1+\ve})^2$ and $C(\ve) = (1+1/\sqrt{1+\ve})^2$.  Theorem \ref{thm:bss2} shows that each $s_i$ is of the form $s_i = l_i b$ for some $l_i \in \mathbb{N}$ and $c_1 c(\ve)^2 \frac{m}{n} \leq b \leq 2 c_1 c(\ve)^2 \frac{m}{n}$, with $c_1 > 0$ being the absolute constant of Theorem \ref{thm:bss2}.
 Therefore, $s_i \geq c_1 c(\ve)^2 \frac{m}{n}$, which implies that
\[
\sum_{i \in J} |\langle v, v_i \rangle |^2 \leq \frac{2 C(\varepsilon)}{c_1 c(\ve)^2} \frac{n}{m} \| v \|^2 \leq \frac{8}{c_1 c(\ve)^2} \frac{n}{m} \| v \|^2 \quad \text{for all} \quad v \in \mathbb{C}^n.
\]
For proving the lower frame bound, we use the upper estimate in \eqref{eq:batson2} to obtain, for $j \in J$,
\begin{align*}
s_j \| v_j \|^4  \leq \sum_{i \in J} s_i |\langle v_j, v_i \rangle |^2 \leq 2C(\ve) \| v_j \|^2.
\end{align*}
This implies $s_j \leq 2C(\ve) m / n$, and thus
\[
  \sum_{i \in J} |\langle v, v_i \rangle |^2 \geq \frac{c(\varepsilon)}{4  C(\varepsilon)} \frac{n}{m} \| v \|^2 \geq \frac{c(\varepsilon)}{16 } \frac{n}{m} \| v \|^2 \quad \text{for all} \quad v \in \mathbb{C}^n,
\]
as claimed.
\end{proof}

As in \cite{nitzan2013few, nitzan2016exponential}, we reformulate \Cref{lem:batson} in a form that is more convenient for applications to exponential frames. For an $m \times n$-matrix $M$ and a subset $J \subseteq \{1, ..., m\}$, the $\#J \times n$ submatrix of $M$ whose rows belong to $J$ is denoted by $M(J)$.

\begin{lemma} \label{lem:submatrix}
There exist absolute constants $c, C>0$ with the following property: If
 $M$ is an $m \times n$-matrix that is a submatrix of some orthonormal $m \times m$-matrix and such that all rows of $M$ have equal $2$-norm, then, 
 for every $\varepsilon > 0$, there exists a subset $J \subseteq \{1, ..., m\}$ of cardinality $\#J \leq \lceil (1+\varepsilon) n \rceil$ satisfying
\[
 c \big(1-1/\sqrt{1+\ve} \big)^2  \frac{n}{m} \| v \|^2 \leq \big\| M(J) v \big\|_{\ell^2(J)} \leq C \frac{1}{(1-1/\sqrt{1+\ve})^4} \frac{n}{m} \| v \|^2 
\]
for all $v \in \mathbb{C}^n$.
\end{lemma}
\begin{proof}
Let $\{v_i\}^m_{i = 1}$ be the rows of the $m \times n$-matrix $M$. Since $M$ is a submatrix of an orthonormal matrix, the rows $\{ v_i \}_{i = 1}^m$ form a Parseval frame \eqref{eq:parseval} for $\mathbb{C}^n$. In addition, the assumption that the vectors $\{v_i\}_{i = 1}^m$ have all equal $2$-norm yields that $\| v_i \|^2 = n/m$. Therefore, \Cref{lem:batson} is applicable and yields, given $\ve > 0$, a frame $\{ v_i \}_{i \in J}$ with index set satisfying $\#J \leq \lceil (1+ \ve) n \rceil$ and frame bounds as in \eqref{eq:batson}. Since the submatrix $M(J)$ is the analysis operator of the complex conjugates $\{\overline{v_i}\}_{i \in J}$ of $\{v_i\}_{i \in J}$, which forms a frame with the same frame bounds, the claim follows.
\end{proof}

\section{Exponential frames on the real line} \label{sec:realline}
In this section, we give a direct proof of Theorem \ref{thm:intro}. The proof is modeled on that of \cite[Theorem 1]{nitzan2013few}, see also \cite[Lemma 7]{nitzan2016exponential}. We will use the notation $e_x (t) = e^{2\pi i x t}$ for $x, t \in \mathbb{R}$ and denote by $|\Omega|$ the Lebesgue measure of a set $\Omega \subseteq \mathbb{R}$. 

The following theorem is a version of Theorem \ref{thm:intro} with more explicit frame bounds. 

\begin{theorem}
There exist absolute constants $c, c'> 0$ with the following property:

Given $\varepsilon > 0$ and a compact set $\Omega \subseteq \R$, which is contained in an interval of length $d>0$, there exists $\Lambda \subseteq d^{-1} \mathbb{Z}$ of uniform density $D(\Lambda) \leq (1+\varepsilon) |\Omega|$ such that
\begin{align} \label{eq:int}
 A(\ve) |\Omega| \| f \|^2  \leq \sum_{\lambda \in \Lambda} |\langle f, e_{\lambda} \rangle |^2 \leq B(\ve) |\Omega| \| f \|^2 
\end{align}
for all $f \in L^2 (\Omega)$, where the constants $A(\ve)$ and $B(\ve)$ are given by
\begin{align} \label{eq:frameconstants0}
A(\varepsilon) = c  \big(1 - 1/ \sqrt{1+\varepsilon/4} \big)^2
\quad \text{and} \quad 
B(\varepsilon) = c'  \big(1 - 1/\sqrt{1+\varepsilon/4}\big)^{-4}.
\end{align}
\end{theorem}

\begin{proof}
First, note that if $\{ e_{\lambda} \mathds{1}_{\Omega} \}_{\lambda \in \Lambda}$ is a frame for $L^2 (\Omega)$, then the system $\{ e_{\lambda} \mathds{1}_{\Omega + x} \}_{\lambda \in \Lambda}$  is a frame for $L^2 (\Omega + x)$ with the same frame bounds for any $x \in \mathbb{R}$. As such, it suffices to assume that $\Omega \subseteq [0, d]$ for some $d>0$. Furthermore,  since rescaling a frame $\{ e_{\lambda'} \mathds{1}_{\Omega'} \}_{\lambda' \in \Lambda'}$ for $L^2 (\Omega')$ with $\Lambda' \subseteq \mathbb{Z}$, $\Omega' \subseteq [0,1]$ and frame bounds $A|\Omega'|, B|\Omega'| $ yields a frame $\{ e_{\lambda} \mathds{1}_{d \Omega'} \}_{\lambda \in d^{-1} \Lambda'}$ for $L^2 (d\Omega')$ with $d^{-1} \Lambda' \subseteq d^{-1} \mathbb{Z}$, $d \Omega' \subseteq [0, d]$ and frame bounds $A|d\Omega'|, B|d\Omega'|$, it may be assumed that $\Omega \subseteq [0, 1]$. 

Second, note that every compact set $\Omega \subseteq [0, 1]$ can be covered by a set of the form
\begin{align} \label{eq:special_set}
 \bigcup_{k \in I} \bigg[ \frac{k}{m}, \frac{k+1}{m} \bigg], \quad I \subseteq \{0, ..., m - 1\}, \; \#I = n, 
\end{align}
where $m, n \in \mathbb{N}$ with $m > n$, and that  the measure of such set can be chosen arbitrary close to that of $\Omega$. We choose $\delta > 0$ such that
\begin{align} \label{eq:assumption1}
(1 + \ve/2) (|\Omega|+\delta) \leq (1+\ve) |\Omega|
\end{align}
and select a set $\Omega' \supseteq \Omega$ of the form \eqref{eq:special_set} such that 
\begin{align} \label{eq:assumption2}
|\Omega| \leq |\Omega'| \leq |\Omega| + \delta.
\end{align}
Note that $|\Omega'| = \frac{n}{m}$. Moreover, by choosing $m \in \mathbb{N}$ sufficiently large, we can choose $n \in \mathbb{N}$ such that $\lceil (1+\ve/4) n\rceil \leq (1+\ve/2) n$.

Denote by $F$ the $m \times m$-matrix with entries $F_{j,k} = e^{i 2\pi j k / m}$ for $j,k = 0, ..., m-1$, and recall that $(\sqrt{m})^{-1} F$ is orthonormal. Let $F_I $ denote the $m \times n$-submatrix of $F$ whose columns are indexed by $I$. Then an application of \Cref{lem:submatrix} to $(\sqrt{m})^{-1} F_I$ yields a subset $J \subseteq \{0, ..., m-1\}$ of cardinality $\#J \leq \lceil(1+\varepsilon/4) n\rceil$ such that
\begin{align} \label{eq:fouriersubmatrix}
 c \cdot \big(1-1/\sqrt{1+\varepsilon/4}\big)^2 n \| v \|^2 \leq \big\| F_I (J) v \big\|^2_{\ell^2(J)} \leq C \cdot \frac{1}{\big(1-1/\sqrt{1+\varepsilon/4}\big)^4} n \| v \|^2 
 \end{align}
for all $v \in \mathbb{C}^n$, where $c, C > 0$ are the absolute constants provided by \Cref{lem:submatrix}. To simplify notation, we set $c(\varepsilon) = c\cdot (1-1/\sqrt{1+\varepsilon})^2$ and $C(\varepsilon) = C \cdot (1-1/\sqrt{1+\varepsilon})^{-4}$.

Define the set $\Lambda := \bigcup_{j \in J} (j + m \mathbb{Z})$. Then
\begin{align} \label{eq:densityR}
 D(\Lambda)  = \#J \cdot D(m \mathbb{Z}) \leq   \frac{\lceil (1+\varepsilon/4) n\rceil}{m} \leq (1+\varepsilon/2) |\Omega'|.
\end{align}
We first show that $\{e_{\lambda} \mathds{1}_{\Omega'} \}_{\lambda \in \Lambda}$ is a frame for $L^2 (\Omega')$.

Let $f \in L^2 (\Omega')$ be arbitrary. Since $L^2 (\Omega') = \bigoplus_{k \in I} L^2 (\frac{k}{m} + [0,1/m])$, we can identify 
$f$ with a sequence $(f_{k/m})_{k \in I} \in L^2 (k/m + [0,1/m])$, and write
\begin{align*}
\sum_{\lambda \in \Lambda} | \langle f, e_{\lambda} \rangle |^2 
&= \sum_{j \in J} \sum_{l \in \mathbb{Z}} \bigg| \sum_{k \in I} \langle f_{k/m} , e_{j + ml} \rangle_{L^2 (k/m + [0,1/m])} \bigg|^2 \\
&= \sum_{j \in J} \sum_{l \in \mathbb{Z}} \bigg| \sum_{k \in I} \langle f_{k/m} (\cdot - k/m)  , e_{j + ml} (\cdot - k/m) \rangle_{L^2 (0,1/m)} \bigg|^2 \\
&= \sum_{j \in J} \sum_{l \in \mathbb{Z}} \bigg| \sum_{k \in I} \langle e_{j} (k/m) e_{-j} (\cdot) f_{k/m} (\cdot - k/m)  , e_{ml} \rangle_{L^2 (0,1/m)} \bigg|^2,
\end{align*}
where the last equality used that  $e_{ml} (k/m) = 1$ for $k, l \in \mathbb{Z}$. A combination of these equalities and the fact that $\{ e_{ml} \mathds{1}_{(0,1/m)} \}_{l \in \mathbb{Z}}$ is an orthogonal basis for $L^2 (0,1/m)$, yields 
\begin{align*}
\sum_{\lambda \in \Lambda} | \langle f, e_{\lambda} \rangle |^2 
&= \frac{1}{m} \sum_{j \in J} \bigg\| \sum_{k \in I} e_{j} (k/m) e_{-j} (\cdot)  f_{k/m} (\cdot - k/m) \bigg\|^2_{L^2 (0,1/m)} \\
&= \frac{1}{m} \sum_{j \in J} \bigg\| \sum_{k \in I} e_{j} (k/m)  f_{k/m} (\cdot - k/m) \bigg\|^2_{L^2 (0,1/m)} \\
&= \frac{1}{m} \int_{0}^{1/m} \big\| F_I (J) \big(f_{k/m} (x - k/m) \big)_{k \in I} \big\|^2_{\ell^2 (J)} \; dx.
\end{align*}
Applying \eqref{eq:fouriersubmatrix} therefore yields
\begin{align*}
\sum_{\lambda \in \Lambda} | \langle f, e_{\lambda} \rangle |^2 
&\leq C(\ve/4) \frac{n}{m} \int_{0}^{1/m} \big\| \big(f_{k/m} (x - k/m) \big)_{k\in I} \big\|^2_{\ell^2 (I)} \; dx \\
&= C(\ve/4) |\Omega'| \sum_{k \in I} \int_{0}^{1/m} | f_{k/m} (x - k/m)|^2 \; dx \\
&= C(\ve/4) |\Omega'| \| f \|^2_{L^2 (\Omega')}.
\end{align*}
 Similar arguments yield the lower bound
\[
\sum_{\lambda \in \Lambda}| \langle f, e_{\lambda} \rangle |^2 \geq c(\varepsilon/4) |\Omega'|  \| f \|^2_{L^2(\Omega')}.
\]
Hence, $\{ e_{\lambda} \mathds{1}_{\Omega'} \}_{\lambda \in \Lambda}$ is a frame for $L^2 (\Omega')$ with frame bounds $c(\ve/4) |\Omega'|$ and $C(\ve/4) |\Omega'|$. 

Lastly, note that $\Omega \subseteq \Omega'$, combined with \eqref{eq:assumption1} and \eqref{eq:assumption2}, yields that
\[
c(\varepsilon/4) |\Omega| \| f \|^2_{L^2 (\Omega)} \leq \sum_{\lambda \in \Lambda} | \langle f, e_{\lambda} \rangle |^2 \leq C(\ve/4) |\Omega'| \| f \|^2_{L^2 (\Omega)} 
\leq C(\ve/4) \frac{1+\ve}{1+\ve/2} |\Omega| \| f \|^2_{L^2 (\Omega)}
\]
for all $f \in L^2 (\Omega)$. Since $(1+\ve)/(1+\ve/2) \leq 2$, this shows the frame inequalities \eqref{eq:int} for the absolute constants $c$ and $c' = 2C$.  In addition, a combination of \eqref{eq:assumption1},  \eqref{eq:assumption2} and \eqref{eq:densityR}  shows that
\[
D(\Lambda) \leq (1+\ve/2) |\Omega'| \leq (1+\ve) |\Omega|.
\]
This completes the proof. 
\end{proof}

Theorem \ref{thm:intro} can be rephrased in terms of sampling in Paley-Wiener spaces. For a compact set $\Omega \subseteq \mathbb{R}$, the Paley-Wiener space $\PW_{\Omega}$ consists of all functions $f \in L^2 (\mathbb{R})$ whose Fourier transform
\[
\widehat{f} (\xi) = \int_{\mathbb{R}} f(t) e^{- 2\pi i \xi t} dt, \quad \xi \in \mathbb{R},
\]
has its essential support inside of $\Omega$. Every function in $\PW_{\Omega}$ is continuous, and the Fourier inversion formula gives
\[
f(x) = \int_{\Omega} \widehat{f}(\xi) e^{2\pi i x \xi} \; d\xi = \langle \widehat{f}, e_{x} \rangle_{L^2(\Omega)}, \quad x \in \mathbb{R}.
\]
The following result is therefore a direct consequence of Theorem \ref{thm:intro}.

\begin{corollary} \label{cor:paley}
Given $\varepsilon > 0$ and a compact set $\Omega \subseteq \R$, which is contained in an interval of length $d>0$, there exists $\Lambda \subseteq d^{-1} \mathbb{Z}$ of uniform density $D(\Lambda) \leq (1+\varepsilon) |\Omega|$ such that
\begin{align*} \label{eq:int}
A(\ve) |\Omega| \| f \|_{L^2 (\mathbb{R})}^2 \leq \sum_{\lambda \in \Lambda} |f(\lambda)|^2 \leq B(\ve) |\Omega| \| f \|_{L^2 (\mathbb{R})}^2 
\end{align*}
for all $ \in \PW_{\Omega}$, where the constants $A(\ve)$ and $B(\ve)$ are as in \eqref{eq:frameconstants0}. 
\end{corollary}

\section{Exponential frames on locally compact abelian groups} \label{sec:LCA}
This section is devoted to proving our main result on exponential frames on locally compact abelian groups. Standard references for background on harmonic analysis on locally compact abelian groups are, e.g., \cite{hewitt1963abstract, reiter2000classical, folland2016course}.

\subsection{Necessary density conditions} \label{sec:necdensityLCA}
Let $G$ be a second countable locally compact abelian (LCA) group with dual group $\widehat{G}$. To avoid trivialities we assume that $G$ is noncompact.
We denote by $\mu_G$ a Haar measure on $G$ and always assume that $\mu_{\widehat{G}}$ is the dual measure on $\widehat{G}$ relative to $\mu_G$, i.e., the Haar measure such that Plancherel's formula holds.

A subset $T \subseteq G$ is said to be \emph{uniformly discrete} if there exists an open set $U \subseteq G$ such that the sets $x+U$, $x\in T$, are pairwise disjoint. In particular, any uniform lattice in $G$, i.e., a discrete cocompact subgroup, is uniformly discrete.

Following \cite{grochenig2008landau}, we define a relation on the set of uniformly discrete sets in the following manner.

\begin{definition}
Given two uniformly discrete sets $T$ and $T'$ and nonnegative numbers $\alpha$ and $\alpha'$, we write $\alpha T \preccurlyeq \alpha' T'$ if for every $\ve>0$ there exists a compact set $K \subseteq G$ such that
\[
 (1-\varepsilon) \alpha \#( T \cap L) \le \alpha' \#(T' \cap (K+L))
\]
for all compact sets $L \subseteq G$.
\end{definition}

Using the relation just defined, we next define the notion of uniform densities.

\begin{definition}\label{ud}
Let $H$ be a fixed uniform lattice in $G$. Let $T$ be another uniformly discrete subset of $G$. The {\it lower uniform density} of $T$ with respect to $H$ is defined as
\[
 D^-_{H} (T) := \sup\{ \alpha \in [0, \infty) : \alpha H \preccurlyeq T\}.
\]
The {\it upper uniform density} of $T$ with respect to $H$ is defined as
\[
 D^+_{H} (T) := \inf\{ \alpha \in [0, \infty) : T \preccurlyeq  \alpha H\}.
\]
If  $D^-_{H} (T) = D^+_{H} (T)$, then $T$ is said to have {\it uniform density} $D_H(T)$ with respect to $H$.
\end{definition}

Alternatively, uniform densities can be phrased in terms of the comparison of measures.
Following \cite{grochenig2008landau}, given two locally finite positive measures $\nu$ and $\mu$ on $G$, we write $\nu \preceq \mu$ if for every $\varepsilon > 0$ there exists a compact set $K \subseteq G$ such that
\[
 (1-\varepsilon) \nu (L) \leq \mu(K+L)
\]
for all compact sets $L \subseteq G$. If we choose $\nu$ to be the counting measure $\delta_{T}$ of $T$ and $\mu$ to be the counting measure $\delta_H$ of $H$, or vice versa, then the notion of uniform densities can be restated as
\[
 D^-_{H} (T) = \sup \{ \alpha \in [0, \infty) : \alpha \delta_{H} \preceq \delta_{T} \},
 \qquad
 D^+_{H} (T) := \inf\{ \alpha \in [0, \infty) :  \delta_T \preceq  \alpha \delta_{H} \} .
\]
Furthermore, if the measure $\mu_G$ is normalized such that the fundamental domain of $G/H$ has measure $1$, then we have $\mu_G \preceq \delta_H$ and $\delta_H \preceq \mu_G$. Thus, we have yet another formulation of uniform density with respect to $H$,
\[
 D^-_{H} (T) = \sup \{ \alpha \in [0, \infty) : \alpha \mu_G \preceq \delta_{T} \},
 \qquad
 D^+_{H} (T) := \inf\{ \alpha \in [0, \infty) :  \delta_T \preceq  \alpha \mu_G \} .
\]

Suppose  that the dual group $\widehat{G}$ is compactly generated. By the structure theory of such groups, the group $\widehat G$ is isomorphic to $\R^d \times \Z^n \times K_0$ for some compact group $K_0$, see, e.g., \cite[Theorem 9.8]{hewitt1963abstract}. Consequently, $G$ is of the form $\R^d \times \T^n \times D_0$ for some countable discrete group $D_0$. We will always select as a uniform lattice $H=\Z^d \times \{e\} \times D_0$ as {\bf the reference lattice} used for computing density, where $e$ is the identity element of $\T^n$. Then the annihilator $H^{\perp} := \{ \chi \in \widehat{G} : \chi(x) = 1, \; \forall x \in H \}$ is a uniform lattice in $\widehat{G}$, and we denote by $\Sigma_{H^{\perp}}$ a fundamental domain of $H^{\perp} \subseteq \widehat{G}$. If the characters on $\R^d$ are given by $x\mapsto e^{2\pi i \langle x, \xi \rangle}$ for $\xi\in\R^d$, then we can take $\Sigma_{H^\perp}=[-1/2,1/2)^d \times \{e\} \times K_0$. We assume that the Haar measure $\mu_{\widehat G}$ is normalized such that $\mu_{\widehat{G}} (\Sigma_{H^{\perp}})  = 1$.

Under these assumptions, Gr\"ochenig, Kutyniok, and Seip \cite{grochenig2008landau} showed the following extension of the classical result of Landau \cite{landau1967necessary} on sampling and interpolation sets for the Paley-Wiener space from the setting of $\R^d$ to LCA groups. 
We state an alternative statement phrased in terms of frames and Riesz sequences. Given $g \in G$, we write $e_g \in \widehat{\widehat{G}}$ for $e_g (\chi) = \chi(g)$ with $\chi \in \widehat{G}$.

\begin{theorem}[\cite{grochenig2008landau}] \label{thm:landau}
Let $G$ be a second countable LCA group whose dual group $\widehat{G}$ is compactly generated. Let $H$ be the reference lattice in $G$ and normalize the Haar measure on $\widehat{G}$ so that the fundamental domain of $\widehat{G}/H^{\perp}$ has measure $1$. 

For any compact subset $\Omega \subseteq \widehat{G}$, the following assertions hold:
\begin{enumerate}[(i)]
\item If $T \subseteq G$ is such that $\{ e_{t} \mathds{1}_{\Omega} \}_{t \in T}$ forms a frame for $L^2 (\Omega)$, then
\[
 D_H^- (T) \geq \mu_{\widehat G}(\Omega).
\]
\item 
 If $T \subseteq G$ is such that $\{ e_{t} \mathds{1}_{\Omega} \}_{t \in T}$ forms a Riesz sequence in $L^2 (\Omega)$, then
\[
 D_H^+ (T) \leq \mu_{\widehat G}(\Omega).
\]
\end{enumerate}
\end{theorem}

\subsection{Frames near the critical density}
In \cite{agora2015multi}, Agora, Antezana and Cabrelli showed the existence of frames near the critical density whenever the dual group is compactly generated. The following result complements this result by also providing frame bounds involving the spectrum. 

\begin{theorem} \label{thm:main_lca}
There exist absolute constants $c, c' > 0$ with the following property:

Let $G$ be a second countable LCA group such that its dual $\widehat G$ is compactly generated. Let $H$ be the reference lattice in $G$. Suppose the Haar measure of $\widehat G$ is normalized such that the fundamental domain of $\widehat{G}/H^\perp$ has measure $1$.

Given $\ve > 0$ and a compact set $\Omega \subseteq \widehat{G}$, there exists $T \subseteq G$ of uniform density $D_H(T) \leq (1+\varepsilon) \mu_{\widehat{G}} (\Omega)$ such that
\begin{align} \label{eq:lca}
 A(\varepsilon) \mu_{\widehat G}(\Omega)  \| f \|^2 \leq \sum_{t \in T} |\langle f, e_t \rangle |^2 \leq B(\varepsilon) \mu_{\widehat G} (\Omega) \| f \|^2 \qquad \text{for all } f \in L^2 (\Omega; \mu_{\widehat G}),
\end{align}
where the constants $A(\ve)$ and $B(\ve)$ are given by
\begin{align} \label{eq:frameconstants}
A(\varepsilon) = c  \big(1 - 1/ \sqrt{1+\varepsilon/4} \big)^2
\quad \text{and} \quad 
B(\varepsilon) = c' \big(1 - 1/\sqrt{1+\varepsilon/4}\big)^{-4}.
\end{align}
\end{theorem}

To prove Theorem \ref{thm:main_lca}, we will construct a frame for an associated elemental quotient group and then lift this frame to the setting of the theorem. We will consider both steps in the following subsections. The proof of Theorem \ref{thm:main_lca} will be given in Section \ref{sec:quasidyadic}.

\subsection{Elemental groups}
We start by proving a version of Theorem \ref{thm:main_lca} under the additional assumption that the group $G$ is an elemental group, which means that it is isomorphic to $\mathbb{R}^d \times \mathbb{T}^n \times \mathbb{Z}^{\ell} \times F$, where $d, n$ and $\ell$ are nonnegative integers, $\T=\R/\Z$, and $F$ is a finite abelian group.

\begin{theorem} \label{thm:ele}
Let $c, C > 0$ be the absolute constants of Lemma \ref{lem:submatrix}.

Let $G=\mathbb{R}^d \times \mathbb{T}^n \times \mathbb{Z}^{\ell} \times F$, where $d, n$ and $\ell$ are nonnegative integers and $F$ is a finite abelian group.  Let 
\[
H = \mathbb{Z}^d \times \{e\}  \times \mathbb{Z}^{\ell} \times F,
\]
be the reference lattice in $G$. Let $\mu_{\widehat G}$ be the normalized Haar measure of $\widehat G$ so that the fundamental domain of $\widehat{G}/H^\perp$ has measure $1$.
For $m \in \mathbb{N}$, define the lattice 
\begin{equation}\label{el2}
H_m=2^m \Z^d \times \{e\} \times 2^m \Z^{\ell} \times \{e\} \subseteq G,
\end{equation}
where $e$ denotes either the identity element of the group $\T^n$ or $F$. Let
\begin{equation}\label{el6}
\Sigma_m = [0, 2^{-m})^d \times \{0\} \times [0, 2^{-m})^{\ell} \times \{e\}.
\end{equation}

Let $\varepsilon > 0$. For any compact set of the form
\begin{equation}\label{el4}
\Omega = \bigcup_{i = 1}^k (\lambda_i + \Sigma_m), \quad \lambda_1, ..., \lambda_k \in {H_m}^{\perp},
\end{equation}
 there exists a set $T$ of the form
\begin{equation}\label{el5}
T=\bigcup_{j=1}^q ( h_j + H_m ), \qquad h_1,\ldots,h_q \in G,
\end{equation}
with uniform density $D_H(T) \leq  \lceil(1+\varepsilon)k\rceil /k \cdot \mu_{\widehat{G}} (\Omega)$ and such that 
\begin{align} \label{eq:intro}
c (1-1/\sqrt{1+\ve})^2 \mu_{\widehat G}(\Omega)  \| f \|^2 \leq \sum_{t \in T} |\langle f, e_t \rangle |^2 \leq C \frac{1}{(1-1/\sqrt{1+\ve})^4}  \mu_{\widehat G} (\Omega) \| f \|^2
\end{align}
for all $ f \in L^2 (\Omega; \mu_{\widehat G})$.
\end{theorem}

\begin{proof}
Since $G= \mathbb{R}^d \times \mathbb{T}^n \times \mathbb{Z}^{\ell} \times F$, its dual is of the form $\widehat{G}=\mathbb{R}^d \times \mathbb{Z}^n \times \mathbb{T}^{\ell} \times \widehat{F}$. The duality pairing between 
$\xi=(\xi_1,\ldots,\xi_4)\in \widehat G$ and $x=(x_1,\ldots,x_4) \in G$ is given by
\[
\xi(x):= e^{2\pi i\lan x_1,\xi_1\ran} e^{2\pi i \langle x_2,\xi_2 \ran} e^{2\pi i \langle x_3,\xi_3 \ran} \xi_4(x_4).
\]
where $\xi_4(x_4)$ denotes the value of the character $\xi_4\in \widehat F$ at $x_4 \in F$.
Thus, the annihilator of $H$ is
\[
H^\perp=\Z^d \times \Z^n \times \{e\} \times \{e\},
\]
where $e$ denotes either the identity element of the group $\T^{\ell}$ or $\widehat F$. Consequently, the set
\[
\Sigma_{H^\perp}=[-1/2,1/2)^d \times \{0\} \times \T^{\ell} \times \widehat F
\]
is a fundamental domain of $\widehat G/H^{\perp}$. Recall that we assume that the Haar measure $\mu_{\widehat G}$ is normalized such that $\mu_{\widehat{G}} (\Sigma_{H^{\perp}})  = 1$.

For $m \in \mathbb{N}$, recall the lattice 
\[
H_m=2^m \Z^d \times \{e\} \times 2^m \Z^{\ell} \times \{e\} \subseteq G,
\]
where $e$ denotes either the identity element of the group $\T^n$ or $F$. 
Its annihilator is the lattice
\[
{H_m}^\perp = 2^{-m} \mathbb{Z}^d \times \mathbb{Z}^n \times \mathbb{Z}_{2^m}^\ell \times \widehat F \subseteq \widehat G,
\]
where $\Z_{2^m}=\{z\in \T=\R/\Z: {2^m}z \equiv 0 \mod 1\}$. Then,
\[
\Sigma_m = [0, 2^{-m})^d \times \{0\} \times [0, 2^{-m})^{\ell} \times \{e\}
\]
is a fundamental domain of $\widehat G/ {H_m}^\perp$.

Since $\Omega \subseteq \widehat{G}$ is compact, there exists $N \in \mathbb{N}$ such that
\begin{equation}\label{el3}
 \Omega \subseteq \Xi:=[-2^{N-1}, 2^{N-1})^d \times [-2^{N-1}, 2^{N-1}) \mathbb{Z}^n \times \mathbb{T}^{\ell} \times \widehat F.
\end{equation}
Define the lattice 
\[
H_0=2^{-N} \Z^d \times \Z_{2^{N}}^n \times \Z^{\ell} \times F \subseteq G,
\]
where $\Z_{2^N}=\{z\in \T=\R/\Z: {2^N}z \equiv 0 \mod 1\}$. Thus, the annihilator of $H_0$ is
${H_0}^\perp = 2^N \Z^d \times 2^N \Z^n \times \{e\} \times \{e\}$. Consequently, $\Xi$  is the fundamental domain of $\widehat G/{H_0}^\perp$. 

Since $\widehat{H_0}$ is isomorphic with $\widehat{G}/{H_0}^\perp$, the family of exponentials $\{ e_h \mathds{1}_{\Xi} \}_{h \in H_0}$ is an orthogonal basis of $L^2 (\Xi)$. Likewise, the family  $\{ e_h \mathds{1}_{\Sigma_m} \}_{h \in H_m}$ is an orthogonal basis of  $L^2 (\Sigma_m)$.  Note that the order of $H_0 / H_m$ is
$ M: = 2^{(N+m)d}  2^{Nn}  2^{m\ell}  \# F$. Let $(h_i)_{i = 1}^M$ be a set of representatives of the cosets $H_0/H_m$. 

Without loss of generality we can assume that points $\lambda_1,\ldots, \lambda_k \in {H_m}^\perp$ in \eqref{el4} are distinct. By \eqref{el6} and \eqref{el3}, we can extend them to a sequence $(\lambda_i)_{i = 1}^M$  of points in ${H_m}^\perp$ such that $\Xi = \bigcup_{i = 1}^M (\lambda_i + \Sigma_m)$. This implies that  $(\lambda_i)_{i = 1}^M$ is a set of representatives of cosets of ${H_m}^\perp/{H_0}^\perp$. Define $\mathcal{F}$ to be the $M \times M$-matrix with entries $e_{h_i} (\lambda_i)$ for $i, j = 1, ..., M$.
Since the dual of the quotient group $H_0/H_m$ is ${H_m}^\perp/{H_0}^\perp$ and characters on a finite group are orthogonal, the matrix $\frac{1}{\sqrt{M}} \mathcal{F}$ is orthonormal. 

Denote the $M \times k$-submatrix of $\mathcal{F}$ whose columns are indexed by $\lambda_i$, with $i \in I := \{1, ..., k\}$, by $\mathcal{F}_I$. Then, by Lemma \ref{lem:submatrix}, there exists $J \subseteq \{ 1, ..., M \}$ of cardinality $\#J \leq \lceil(1+\varepsilon) k\rceil$ and absolute constants $c, C > 0$ such that
\begin{align} \label{eq:bss_lca}
c (1-1/\sqrt{1+\ve})^2 \frac{k}{M} \| v \|^2 \leq \frac{1}{M} \big\| \mathcal{F}_I (J) v \big\|^2 \leq C \frac{1}{(1-1/\sqrt{1+\ve})^4} \frac{k}{M} \| v \|^2
\end{align}
for all $v \in \mathbb{C}^k$, where $\mathcal{F}_I (J)$ denotes the $|J| \times k$ submatrix of $\mathcal F_I$ with rows indexed by the subset $J$.

Define the index set $T := \bigcup_{j\in J} ( h_j + H_m )$. By reindexing points $h_j$ the set $T$ is of the form \eqref{el5} with $q=\# J$. For showing that $\{ e_t \mathds{1}_{\Omega} \}_{t \in T}$ is a frame for $L^2 (\Omega)$,
note that
$L^2 (\Omega) = \bigoplus_{i = 1}^k L^2 (\lambda_i + \Sigma_m)$, and thus each $f \in L^2 (\Omega)$ can be identified with a sequence $(f_{\lambda_i} )_{i = 1}^k$ with $f_{\lambda_i} \in L^2 (\lambda_i + \Sigma_m)$, i.e., $f_{\lambda_i} (\cdot - \lambda_i) \in L^2 (\Sigma_m)$.
Hence, for $f \in L^2 (\Omega) = \bigoplus_{i = 1}^k L^2 (\lambda_i + \Sigma_m)$ we have
\begin{align*}
 \sum_{t \in T} \big|\langle f, e_t \rangle_{L^2 (\Omega)} \big|^2 &= \sum_{j \in J} \sum_{h \in H_m} \bigg| \sum_{i \in I} \langle f_{\lambda_i} , e_{h + h_j} \rangle_{L^2 (\lambda_i + \Sigma_m)} \bigg|^2 \\
 &= \sum_{j \in J} \sum_{h \in H_m} \bigg| \sum_{i \in I} \langle f_{\lambda_i} (\cdot - \lambda_i) , e_{h + h_j} (\cdot - \lambda_i) \rangle_{L^2 (\Sigma_m)} \bigg|^2.
\end{align*}
Note that
\[
 \big\langle f_{\lambda_i} (\cdot - \lambda_i), e_{h + h_j} (\cdot - \lambda_i) \big\rangle_{L^2 (\Sigma_m)} = e_{h} (\lambda_i) e_{h_j}(\lambda_i) \big\langle e_{-h_j} (\cdot) f_{\lambda_i} (\cdot - \lambda_i), e_{h}  \big\rangle_{L^2 (\Sigma_m)}
\]
and $e_{h} (\lambda_i) = 1$ for $h \in H_m$. Since $\{ e_h \mathds{1}_{\Sigma_m} \}_{h \in H_m}$ is an orthogonal basis of  $L^2 (\Sigma_m)$ we have, for fixed $j \in J$,
\begin{align*}
\sum_{h \in H_m}  & \bigg| \sum_{i \in I}  \big\langle f_{\lambda_i} (\cdot - \lambda_i),  e_{h + h_j} (\cdot - \lambda_i) \big\rangle_{L^2 (\Sigma_m)} \bigg|^2 \\
&=
\sum_{h \in H_m} \bigg| \sum_{i\in I} \big\langle e_{h_j} (\lambda_i) e_{-h_j} (\cdot) f_{\lambda_i} (\cdot - \lambda_i), e_{h} \big\rangle_{L^2 (\Sigma_m)} \bigg|^2 \\
&=  \mu_{\widehat G}(\Sigma_m) \bigg\| \sum_{i \in I} e_{h_j} (\lambda_i) e_{-h_j} (\cdot) f_{\lambda_i} (\cdot - \lambda_i) \bigg\|^2_{L^2(\Sigma_m)} \\
&= \mu_{\widehat G}(\Sigma_m)  \bigg\| \sum_{i \in I} e_{h_j} (\lambda_i) f_{\lambda_i} (\cdot - \lambda_i) \bigg\|^2_{L^2(\Sigma_m)}.
\end{align*}
Therefore, for $f \in L^2 (\Omega)$,
\[
 \sum_{t \in T} |\langle f, e_t \rangle |^2 = \mu_{\widehat G}(\Sigma_m) \int_{\Sigma_m} \bigg \| \mathcal{F}_I (J) \big(f_{\lambda_i} (x - \lambda_i) \big)_{i = 1}^k \bigg\|_{\ell^2 (J)}^2 \; d\mu_{\widehat G}(x),
\]
where $\mathcal{F}_I (J)$ denotes the $ |J| \times k$ matrix with entries $e_{h_j} (\lambda_i)$ for $j \in J$ and $i \in I$. Thus, using the inequalities \eqref{eq:bss_lca} yields
\begin{align*}
 \sum_{t \in T} | \langle f, e_t \rangle |^2 &\leq C \frac{1}{(1-1/\sqrt{1+\ve})^4}  k \mu_{\widehat G}(\Sigma_m) \int_{\Sigma_m} \big \| \big(f_{\lambda_i} (x - \lambda_i) \big)_{i = 1}^k \big\|_{\ell^2 (I)}^2 \; d\mu_{\widehat G}(x) \\
 &= C \frac{1}{(1-1/\sqrt{1+\ve})^4} \mu_{\widehat G}(\Omega) \int_{\Omega} | f (x) |^2 \; d\mu_{\widehat G}(x),
\end{align*}
and similarly 
\[
\sum_{t \in T} | \langle f, e_t \rangle |^2 \geq c (1-1/\sqrt{1+\ve})^2 \mu_{\widehat G}(\Omega)  \| f \|_{L^2 (\Omega)}^2.
\]
This shows that \eqref{eq:intro} holds. 

Since $D_H(H_m)= (2^{md+m\ell} \#F)^{-1}$, we have 
\[
D_H(T)= \#J \cdot D_H(H_m)=\#J \cdot (2^{md+m\ell} \# F)^{-1}.
\]
Since $1=\mu_{\widehat G}(\Sigma_{H^\perp})= 2^{md+m\ell} \#F \mu_{\widehat G}(\Sigma_m)$, we have 
\[
\mu_{\widehat G}(\Omega) = k  \mu_{\widehat G}(\Sigma_m) =k  (2^{md+m\ell} \# F)^{-1}.
\]
Thus, $D_H(T)= (\# J /k) \mu_{\widehat G}(\Omega) \le  \lceil(1+ \ve) k \rceil /k \cdot \mu_{\widehat G}(\Omega)$,
which completes the proof.
\end{proof}

\subsection{Lifting property of frames}
In this subsection, we show a lifting property for frames on quotient groups. A similar result for Riesz bases (resp. orthonormal bases) was shown in \cite[Theorem 5.5]{agora2015multi} (resp. \cite[Lemma 3]{grochenig2008landau}).
The validity of such a result for frames was claimed in \cite[Remark 5.6]{agora2015multi}, albeit without the proof, which we will provide.

Given a closed normal subgroup $K$ of a second countable locally compact group $G$, we say that the Haar measures $\mu_G$, $\mu_K$ and $\mu_{G/K}$ on $G$, $K$ and $G/K$, respectively, satisfy \emph{Weil's integral formula} if they satisfy the identity
\begin{align} \label{eq:weil}
\int_G f(x) \; d\mu_G (x) = \int_{G/K} \int_K f(x + k) \; d\mu_{K}(k) d\mu_{G/K} (x+K)
\end{align}
for $f \in L^1 (G)$. If two out of the three measures $\mu_G$, $\mu_K$ and $\mu_{G/K}$ are given, then the third one can be normalized such that \eqref{eq:weil} holds, see, e.g., \cite[Theorem 3.4.6]{reiter2000classical}.

The following proposition provides a proof of the claim in \cite[Remark 5.6]{agora2015multi}.

\begin{proposition} \label{lift}
Let $K \subseteq G$ be a compact subgroup of a second countable locally compact abelian group $G$. Let $\mu_K, \mu_{G/K}$ and $\mu_{G}$ be Haar measures on $K, G/K$ and $G$, respectively, so that Weil's integral formula \eqref{eq:weil} holds and $\mu_K$ is a normalized measure.

Suppose $Q \subseteq G / K$ is a Borel set such that a sequence $ \{\gamma_n\}$ of characters in $\widehat{G/K} \cong K^{\perp} \subseteq \widehat{G}$ is a frame for $L^2 (Q)$ with frame bounds $A, B > 0$. Let $\{ \kappa_m \}$ be representatives for the cosets $\widehat{G}/K^{\perp} \cong \widehat{K}$. Then $\{ \gamma_n + \kappa_m \}_{m,n}$ is a frame for $L^2 (\pi^{-1} (Q))$ with frame bounds $A, B$, where $\pi : G \to G/K$ is the canonical projection.
\end{proposition}
\begin{proof}
Let $f \in L^2 (\pi^{-1} (Q))$. Using Weil's integral formula \eqref{eq:weil}, a direct calculation gives, for arbitrary $\gamma_n, \kappa_m$,
\begin{align*}
 \langle f,\gamma_n + \kappa_m \rangle
 &= \int_{\pi^{-1} (Q)} f(x) \overline{(\gamma_n + \kappa_m)(x)} \; d\mu_{G} (x) \\
 &= \int_{Q} \int_K f(x + k) \overline{(\gamma_n + \kappa_m)(x+k)} \; d\mu_K (k) d\mu_{G / K} (\pi(x)) \\
 &=  \int_{Q} \int_K f(x + k) \overline{\kappa_m(x+k)} \; d\mu_K (k) \overline{ \gamma_n (\pi(x))} \; d\mu_{G / K} (\pi(x)),
\end{align*}
where we used that $\gamma_n (k) = 1$ for all $k \in K$. For fixed $\kappa_m$, define the function $[f, \kappa_m] : G/K \to \mathbb{C}$ by
\[
 [f, \kappa_m] (\pi(x)) = \int_K f(x + k) \overline{\kappa_m(x+k)} \; d\mu_K (k)
\]
so that
\[
 \langle f, \gamma_n + \kappa_m \rangle = \int_{Q} [f, \kappa_m ] (\pi(x)) \overline{ \gamma_n(\pi(x))} \; d\mu_{G / K} (\pi(x)).
\]
Using the assumption that $\{\gamma_n (\pi(x)) \}_{n}$ is a frame for $L^2 (Q)$ with bounds $A, B > 0$,
it follows that, for fixed $\kappa_m$,
\begin{align*}
 A \int_{Q} \big| [f, \kappa_m] (\pi(x)) \big|^2 \; d\mu_{G/K} (\pi(x))
 & \le
 \sum_{n} |\langle f, \gamma_n + \kappa_m \rangle |^2 \\
 & \leq B \int_{Q} \big| [f, \kappa_m] (\pi(x)) \big|^2 \; d\mu_{G/K} (\pi(x)).
\end{align*}
It remains therefore to show that
\[
 \sum_{m} \int_{Q} \big| [f, \kappa_m] (\pi(x)) \big|^2 \; d\mu_{G/K} (\pi(x)) =  \| f \|_{L^2 (\pi^{-1}(Q))}^2.
\]
For this, note that since $\{ \kappa_m \}_m$ is an orthonormal basis for $L^2(K, \mu_K)$,
a direct calculation gives
\begin{align*}
\int_{Q} \sum_{m} \big| [f, \kappa_m] (\pi(x)) \big|^2  \; d\mu_{G/K} (\pi(x))
&= \int_{Q} \sum_{m} \bigg| \int_K f(x + k) \overline{\kappa_m(k)} \; d\mu_K (k) \bigg|^2 \; d\mu_{G/K} (\pi(x)) \\
&= \int_{Q} \int_K | f(x + k)|^2 \; d\mu_K (k) d\mu_{G/K} (\pi(x)) \\
&=  \| f \|_{L^2 (\pi^{-1}(Q))}^2,
\end{align*}
where the last equality follows from another application of Weil's integral formula \eqref{eq:weil}. This completes the proof.
\end{proof}

\subsection{Quasi-dyadic cubes} \label{sec:quasidyadic}
In this subsection, we will assume that the dual group $\widehat{G}$ of $G$ is compactly generated. By the structure theory of such groups, $\widehat{G}$ can then be identified with $\mathbb{R}^d \times \mathbb{Z}^n \times K_0$ for nonnegative integers $d, n$ and a compact group $K_0$.
Under this assumption, we can exploit the following classical structural result, see, e.g., \cite[Theorem 9.6]{hewitt1963abstract}.

\begin{lemma}
Let $U$ be a unit neighborhood in $\widehat{G}$. There exists a compact subgroup $K \subseteq \widehat{G}$ satisfying $K \subseteq U$ and such that $
 \widehat{G} / K $ is elemental, that is
 \begin{equation}\label{eleme}
 \widehat{G}/K \cong \R^d \times \Z^n \times \T^{\ell} \times \widehat F,
 \end{equation}
where $F$ is a finite abelian group.
\end{lemma}

Agora, Antezana, and Cabrelli \cite{agora2015multi} introduced the concept of quasi-dyadic cubes, which we recall next.

\begin{definition} Let $K$ be a compact subgroup of $\widehat G$ such that $\widehat{G}/K $ is an elemental group of the form \eqref{eleme}. Let $\pi: \widehat{G} \to \widehat{G}/K $ be the canonical projection. For $m\in \N$, define the lattice $\Lambda_m$ in $\widehat{G}/K$ and its fundamental domain $\Sigma_m$ by
\begin{align*}
\Lambda_m &= 2^{-m} \mathbb{Z}^d \times \mathbb{Z}^n \times \mathbb{Z}_{2^m}^\ell \times \widehat F,
\\
\Sigma_m &= [0, 2^{-m})^d \times \{0\} \times [0, 2^{-m})^{\ell} \times \{e\}.
\end{align*}
The family of {\it quasi-dyadic cubes of generation $m$ associated to $K$} is defined as 
\[
Q^{(m)}_\lambda = \pi^{-1}(\lambda+\Sigma_m) \qquad\text{for }\lambda \in \Lambda_m.
\]
\end{definition}

The following proposition was shown in \cite[Proposition 5.3]{agora2015multi}.

\begin{lemma}\label{qd} Let $C$ be a compact set and $V$ an open set such that $C \subseteq V \subseteq \widehat{G}$. There exists a compact subgroup $K$ of $\widehat{G}$ such that $\widehat{G}/K$ is an elemental group, and quasi-dyadic cubes $Q^{(m)}_{\lambda_1}, \ldots, Q^{(m)}_{\lambda_k}$ of sufficiently large generation $m\in \N$ such that
\[
C \subseteq \bigcup_{j=1}^k Q^{(m)}_{\lambda_j} = \pi^{-1}\bigg( \bigcup_{j=1}^k (\lambda_j + \Sigma_m) \bigg) \subseteq V.
\]
\end{lemma}

Using the above results, we can now prove Theorem \ref{thm:main_lca}.

\begin{proof}[Proof of Theorem \ref{thm:main_lca}]
Since the dual group $\widehat{G}$ is compactly generated, by the structure theory of LCA groups, we can identify $\widehat G$ with the group $\R^d \times \Z^n \times K_0$ for some compact group $K_0$.
Thus, $G$ is of the form $\R^d \times \T^n \times D_0$ for some countable discrete group $D_0$. Let $H=\Z^d \times \{e\} \times D_0$ be the reference lattice of $G$, where $e$ is the identity element of $\T^n$. Let $\Sigma_{H^{\perp}} = [-1/2,1/2)^d \times \{e\} \times K_0$ be a fundamental domain of $H^{\perp} \subseteq \widehat{G}$. We assume that the Haar measure $\mu_{\widehat G}$ is normalized such that $\mu_{\widehat{G}} (\Sigma_{H^{\perp}})  = 1$.

Let $\delta>0$ be sufficiently small such that
\begin{equation}\label{m9}
(1+\ve/2) (\mu_{\widehat G}(\Omega)+ \delta) \le (1+\ve) \mu_{\widehat G}(\Omega).
\end{equation}
Since the Haar measure is regular, there exists an open set $V\subseteq \widehat{G}$ such that $\Omega \subseteq V$ and the measure $\mu_{\widehat G}(V \setminus \Omega) < \delta$. By Lemma \ref{qd} there exists
a compact subgroup $K$ of $\widehat{G}$ such that $\widehat{G}/K$ is an elemental group of the form 
\begin{equation}\label{m5}
 \widehat{G}/K \cong \R^d \times \Z^n \times \T^{\ell} \times \widehat F,
\end{equation}
 where $d, n$ and $\ell$ are nonnegative integers and $F$ is a finite abelian group. Hence, we necessarily have $K \subseteq \{0\} \times \{0\} \times K_0$. In addition, Lemma \ref{qd} yields quasi-dyadic cubes $Q^{(m)}_{\lambda_1}, \ldots, Q^{(m)}_{\lambda_{k}}$ of generation $m \in \N$ such that
\[
\Omega \subseteq \bigcup_{j=1}^{k} Q^{(m)}_{\lambda_j}  \subseteq V.
\]
Moreover, by choosing $m\in \N$ to be sufficiently large, $k$ can be made arbitrarily large so that the inequality $\lceil (1+\ve/4) k\rceil \le (1+\ve/2) k$ holds.
 
 By the Pontryagin duality $G \cong \widehat{\widehat G}$,  we can identify the annihilator $K^\perp$ as a subgroup of $G$ by
\[
K^\perp=\{x\in G: \xi(x)=1 \ \forall \xi \in K\} %= \Z^d \times \Z^n \times \{e\} \times \{e\},
\]
The group isomorphism $\widehat{\widehat{G}/K} \cong K^\perp$ and \eqref{m5} imply that 
\[
\widehat{G}/K \cong \widehat{K^\perp} \cong \R^d \times \Z^n \times \T^{\ell} \times \widehat F .
\]
Hence, $\widehat{G}/K$ is the dual group of  
\[
K^\perp  \cong \R^d \times \T^n \times \Z^{\ell} \times F.
\] 
 The reference lattice for  $K^\perp$ is 
\[
H_0 = \mathbb{Z}^d \times \{e\} \times \Z^{\ell} \times F,
\]
where $e$ is the identity element of the group $\T^n$. We assume that  $\mu_K, \mu_{\widehat G/K}$ and $\mu_{\widehat G}$ are Haar measures on $K, \widehat G/K$ and $\widehat G$, respectively, so that Weil's integral formula \eqref{eq:weil} holds and $\mu_K$ is normalized measure. Let $\Sigma_{{H_0}^{\perp}} =[-1/2,1/2)^d \times \{e\} \times \T^\ell \times \widehat F $ be a fundamental domain of ${H_0}^{\perp} \subseteq \widehat{G}/K$. Let $\pi: \widehat G \to \widehat G/K$ be the canonical projection. Since $\pi(\Sigma_{H^\perp})= \Sigma_{H_0^{\perp}} $, we have the correct normalization of the Haar measure $\mu_{\widehat G/K}(\Sigma_{{H_0}^{\perp}})=1$.

We are now ready to apply Theorem \ref{thm:ele} to the compact set
\[
\Omega_0 := \bigcup_{i=1}^{k} (\lambda_i + \Sigma_{m}) \subseteq \widehat{G}/K.
\]
Applying Theorem \ref{thm:ele} with $\ve$ replaced by $\ve/4$ yields an index set $T_0 \subseteq K^\perp$ of the form
\[
T_0=\bigcup_{j=1}^q ( h_j + H_m ), \qquad h_1,\ldots,h_q \in K^{\perp} \subseteq G,
\]
with uniform density 
\begin{equation}\label{m7}
D_{H_0}(T_0) \leq \lceil(1+\varepsilon/4) k \rceil / k \cdot \mu_{\widehat{G}/K} (\Omega_0)
\le (1+\varepsilon/2) \mu_{\widehat{G}/K} (\Omega_0)
\end{equation}
 such that
\begin{align} \label{m10}
 c(\varepsilon/4) \mu_{\widehat G/K}(\Omega_0)  \| f \|^2 \leq \sum_{t \in T_0} |\langle f, e_t \rangle |^2 \leq C(\varepsilon/4) \mu_{\widehat G/K} (\Omega_0) \| f \|^2 
\end{align}
for all $ f \in L^2 (\Omega_0; \mu_{\widehat G/K})$, where the constants $c(\ve)$ and $C(\ve)$ for $\varepsilon > 0$ are given by
\[
c(\ve) := c \cdot (1-1/\sqrt{1+\varepsilon})^2 \quad \text{and} \quad C(\ve) := C \cdot (1 - 1/\sqrt{1+\varepsilon})^{-4},
\]
respectively. 

For constructing a frame for $L^2 (\Omega)$, let $\{ \kappa_i \}$ be representatives for the cosets $G/K^{\perp} \cong \widehat{K}$, and define
\[
T := \bigcup_i (\kappa_i + T_0).
\]
By Proposition \ref{lift}, the system $\{e_t \mathds{1}_{\pi^{-1} (\Omega_0)} \}_{t\in T}$ is a frame for $L^2 (\pi^{-1} (\Omega_0))$ with the same frame bounds as in \eqref{m10}. Since $\Omega \subseteq \pi^{-1}(\Omega_0) \subseteq V$ we have
\[
\mu_{\widehat G}(\Omega) \le \mu_{\widehat G}(\pi^{-1}(\Omega_0)) = \mu_{\widehat G/K}(\Omega_0) \le \mu_{\widehat G}(\Omega) +\delta.
\]
Hence, for all $f \in L^2 (\Omega; \mu_{\widehat G})$,
\begin{align} \label{m12}
 c(\varepsilon/4) \mu_{\widehat G}(\Omega)  \| f \|^2 \leq \sum_{t \in T} |\langle f, e_t \rangle |^2 \leq C(\varepsilon/4) (\mu_{\widehat G} (\Omega)+\delta) \| f \|^2.
\end{align}
By \eqref{m9} this yields the frame inequalities \eqref{eq:lca} with 
constants
\[
A(\varepsilon) := c \cdot \big(1 - 1/ \sqrt{1+\varepsilon/4} \big)^2
\]
and 
\[
B(\varepsilon) := c' \cdot (1 - 1/\sqrt{1-\varepsilon/4})^{-4}
\]
for $c' = 2C$. 

Note that the reference lattice $H_0$ of the group $K^\perp$ is a disjoint union of $2^{m(d+\ell)}$ translates of the lattice $H_m$ given by \eqref{el2}. Hence, $D_{H_0}(H_m)= 2^{-m(d+\ell)}$. Consequently, $D_{H_0}(T_0)= q2^{-m(d+\ell)}$. Likewise, the reference lattice $H=\bar\pi^{-1}(H_0)$ of the group $G$ is a disjoint union of $2^{m(d+\ell)}$ translates of the lattice $\bar\pi^{-1}(H_m)$, where $\bar \pi: G \to G/K^\perp$ is the canonical projection.  Hence, $D_H(T)=q D_H(\bar \pi^{-1}(H_m))= q 2^{-m(d+\ell)}$. This shows that the uniform density of $T_0$ is the same as $T=\bar \pi^{-1}(T_0)$, that is,
$
D_H(T)= D_{H_0}(T_0)$.
Therefore, by \eqref{m9} and \eqref{m7} we have
\[
\begin{aligned}
D_H(T)= D_{H_0}(T_0) &\leq (1+\varepsilon/2) \mu_{\widehat{G}/K} (\Omega_0)
= (1+\varepsilon/2) \mu_{\widehat{G}} (\pi^{-1}(\Omega_0)) \\
&\le (1+\ve/2) (\mu_{\widehat G}(\Omega) +\delta)
\le (1+\ve) \mu_{\widehat G}(\Omega).
\end{aligned}
\]
This completes the proof of Theorem \ref{thm:main_lca}.
\end{proof}

\subsection{Beurling and Leptin densities} \label{sec:beurling}
In this subsection, we discuss the relation between the uniform densities, which were defined in \cite{grochenig2008landau}, and the notions of Beurling and Leptin densities considered in the papers \cite{richard2020on, pogorzelski2022leptin, enstad2024dynamical}. The latter notions of densities allow us to prove a version of Theorem \ref{thm:main_lca} without the assumption that the dual group is compactly generated.

We assume that $G$ is a second countable locally compact group with left Haar measure $\mu_G$. Although $G$ might be nonabelian, we continue to write the group law additively.

We start by recalling notions of F\o lner sequences. A sequence of nonnull compact sets $K_n \subseteq G$ is said to be a \emph{F\o lner sequence} in $G$ if it satisfies, for all compact sets $K \subseteq G$,
\[
\lim_{n \to \infty} \frac{\mu_G ((K+K_n) \Delta K_n )}{\mu_G (K_n)} = 0, 
\]
where $(K_n + K) \Delta K_n$ denotes the symmetric difference between $K_n +K$ and $K_n$. For a general second countable locally compact group $G$, the existence of a F\o lner sequence is one of the many characterizations of the group $G$ being \emph{amenable} \cite{paterson1988amenability}; see \cite[Theorem 4 and Corollary 6]{emerson1968ratio}. In particular, any abelian group $G$ admits a F\o lner sequence; a construction of such a sequence is given in (the proof of) \cite[Theorem 5]{emerson1968ratio}.
A F\o lner sequence $(K_n)_{n \in \mathbb{N}}$ is said to be a \emph{strong F\o lner sequence} if it satisfies the stronger condition
\begin{align} \label{eq:strongfolner}
 \lim_{n \to \infty} \frac{\mu_G((K + K_n) \cap (K + K_n^c ))}{\mu_G(K_n)} = 0
\end{align}
for all compact sets $K \subseteq G$. 
If $(K_n)_{n \in \mathbb{N}}$ is a F\o lner sequence and $L \subseteq G$ is a compact symmetric unit neighborhood, then $(L + K_n )_{n \in \mathbb{N}}$ is a strong F\o lner sequence in $G$, see, e.g., 
\cite[Proposition 5.10]{pogorzelski2022leptin}. 

Following the work of Pogorzelski, Richard, and Strungaru \cite{pogorzelski2022leptin}, given any strong F\o lner sequence $(K_n)_{n \in \mathbb{N}}$ in $G$, the associated \emph{lower Beurling density} of a discrete set $T \subseteq G$ is  defined as
\[
 D^- (T) := \liminf_{n \to \infty} \inf_{x \in G} \frac{\#(T \cap (K_n+x))}{\mu_G(K_n)}
\]
and the \emph{upper Beurling density} of $T$ is
\[
 D^+ (T) := \limsup_{n \to \infty} \sup_{x \in G} \frac{\#(T \cap (K_n+x))}{\mu_G(K_n)}
\]
The significance of the use of strong F\o lner sequences in defining Beurling densities is that the densities can be shown to be independent of such a sequence whenever the group $G$ is unimodular, meaning that left Haar measure is also right-invariant. Before stating this result, we also recall the Leptin densities introduced in \cite{pogorzelski2022leptin}. 

The \emph{lower} and \emph{upper Leptin density} of a uniformly separated set $T$ are defined by
\[
 L^- (T) := \sup_{K \in \mathcal{K}} \inf_{K' \in \mathcal{K}_p} \frac{\# (T \cap (K + K'))}{\mu_G (K')}
\]
respectively
\[
L^+ (T) := \inf_{K \in \mathcal{K}} \sup_{K' \in \mathcal{K}_p} \frac{\# (T \cap K')}{\mu_G (K + K')},
\]
where $\mathcal{K}$ (resp. $\mathcal{K}_p$) denotes the set of nonempty compact sets (resp. the set of compact sets of positive Haar measure). The relation between Beurling densities and Leptin densities is given by the following result, cf. \cite[Proposition 5.14]{pogorzelski2022leptin}.

\begin{lemma}[\cite{pogorzelski2022leptin}] \label{lem:beurlingleptin}
Let $G$ be a second countable unimodular locally compact group admitting a (strong) F\o lner sequence.
 Let $T \subseteq G$ be a discrete set. Then
 \[
  D^-(T) = L^- (T) \quad \text{and} \quad D^+ (T) = L^+(T).
 \]
In particular, the Beurling densities are independent of the choice of strong F\o lner sequence.
\end{lemma}

We will say that a discrete set $T$ has uniform Beurling (resp. Leptin) density if $D^- (T) = D^+ (T)$ (resp. $L^- (T) = L^+ (T)$), and denote its density by $D(T)$ (resp. $L(T)$). If $H$ is a uniform lattice in $G$ with fundamental domain $Q_H$, then it has uniform Leptin density
\[
 L(H) = \frac{1}{\mu_{G} (Q_H)},
\]
cf. \cite[Proposition 6.1]{pogorzelski2022leptin}.

In addition to the previous lemma, we will also use a relation between uniform densities (see Section \ref{sec:necdensityLCA}) and Leptin densities. As for abelian groups, given two counting measures $\delta_T$ and $\delta_H$ of discrete subsets $T$ and $H$ of $G$, we write $\delta_T \preceq \delta_H$ if for every $\ve > 0$, there exists a compact set $K \subseteq G$ such that
\[
(1-\ve) \delta_T (K) \leq \delta_H (K + L)
\]
for all compact sets $L \subseteq G$. Using this notation, the following result corresponds to \cite[Proposition C.2]{pogorzelski2022leptin}.

\begin{lemma}[\cite{pogorzelski2022leptin}] \label{lem:densities}
 Let $H \subseteq G$ be a uniform lattice with fundamental domain $Q_H \subseteq G$. If $T \subseteq G$ is a discrete set, then
 \[
 \frac{ \sup \{\alpha \in [0, \infty) : \alpha \delta_H \preceq \delta_T \}}{\mu_G (Q_H)} = L^- (T) \quad \text{and} \quad \frac{ \inf \{\alpha \in [0, \infty) :  \delta_T \preceq \alpha \delta_H \} }{\mu_G (Q_H)} = L^+ (T),
 \]
 where  $L^-(T)$ and $L^+(T)$ denote the lower and upper Leptin densities of $T$.
\end{lemma}

Using Lemma \ref{lem:beurlingleptin} and Lemma \ref{lem:densities}, the necessary density conditions of Theorem \ref{thm:landau} can be rephrased in terms of Beurling densities. This has the advantage that it allows us to remove the assumption that the dual group is compactly generated. Moreover, this version does not depend on normalization of Haar measures.  The following result was stated as \cite[Fact 1.7]{richard2020on} and \cite[Theorem 4.2]{enstad2024dynamical}.

\begin{theorem} \label{thm:landau_lca}
Let $G$ be a second countable LCA group with dual group $\widehat{G}$. Let $\Omega \subseteq \widehat{G}$ be a compact subset. Then,
\begin{enumerate}[(i)]
\item If $T \subseteq G$ is such that $\{ e_{t} \mathds{1}_{\Omega} \}_{t \in T}$ forms a frame for $L^2 (\Omega)$, then
\[
 D^- (T) \geq \mu_{\widehat G}(\Omega).
\]
\item 
 If $T \subseteq G$ is such that $\{ e_{t} \mathds{1}_{\Omega} \}_{t \in T}$ forms a Riesz sequence in $L^2 (\Omega)$, then
\[
 D^+ (T) \leq \mu_{\widehat G}(\Omega).
\]
\end{enumerate}
\end{theorem}

Lastly, we rephrase Theorem \ref{thm:main_lca} in terms of Beurling densities and remove the assumption that the dual group is compactly generated.
The following two lemmata will be used for this; see \cite[Lemma 9]{grochenig2008landau} for the first lemma.

\begin{lemma}[\cite{grochenig2008landau}] \label{lem:notcompactlygenerated}
Let $G$ be a second countable LCA group with dual group $\widehat{G}$. 
Suppose that $\Omega \subseteq \widehat{G}$ is compact and let $\Gamma$ be the open subgroup in $\widehat{G}$ generated by $\Omega$. Then $\Gamma$ is compactly generated, $\Gamma^{\perp}$ is compact and every function in 
\[
\PW_{\Omega} (G) := \big\{ f \in L^2 (G) : \supp \widehat{f} \subseteq \Omega \big\}
\]
 is  $\Gamma^{\perp}$-periodic. Furthermore,  $\widehat{G/\Gamma^{\perp}} = \Gamma$.
\end{lemma}

\begin{lemma} \label{lem:folnerquotient}
Let $G$ be a second countable locally compact group and let $K \subseteq G$ be a compact normal subgroup. 
If $(K_n)_{n \in \mathbb{N}}$ is a strong F\o lner sequence in $G/K$, then $(\pi^{-1} (K_n))_{n \in \mathbb{N}}$ is a strong F\o lner sequence in $G$, where $\pi : G \to G/K$ is the canonical projection. 
\end{lemma}
\begin{proof}
Since $K$ is compact, the canonical projection $\pi : G \to G/K$ is proper, meaning that preimages of compact sets are compact. In particular, 
each preimage $ \pi^{-1} (K_n)$, $n \in \mathbb{N}$, is compact and nonnull. To verify that $(\pi^{-1} (K_n))_{n \in \mathbb{N}}$ satisfies the strong F\o lner property \eqref{eq:strongfolner}, let $L \subseteq G$ be compact. Then
\begin{align*}
\big( L + \pi^{-1} (K_n) \big) \cap \big( L + (\pi^{-1} (K_n))^c \big) 
& = \big( L + \pi^{-1} (K_n) \big) \cap \big( L + \pi^{-1} (K_n^c) \big) \\
& \subseteq \pi^{-1} \big( \pi(L) + K_n \big) \cap \pi^{-1} \big(\pi(L) + K_n^c \big) \\
& = \pi^{-1} \big( \big( \pi(L) + K_n \big) \cap \big(\pi(L) + K_n^c \big) \big).
\end{align*}
If $\mu_G$ is a Haar measure on $G$, then $\mu_{G/K} := \mu_{G} (\pi^{-1} (\cdot))$ is a Haar measure on $G/K$ by Weil's integral formula \eqref{eq:weil}. Therefore, 
\begin{align*}
\frac{\mu_G \big( \big( L + \pi^{-1} (K_n) \big) \cap \big( L + (\pi^{-1} (K_n))^c \big)\big) }{\mu_G (\pi^{-1} (K_n))} 
&\leq \frac{\mu_G \big( \pi^{-1} \big( \big( \pi(L) + K_n \big) \cap \big(\pi(L) + K_n^c \big) \big)\big) }{\mu_G (\pi^{-1} (K_n))} \\
&= \frac{\mu_{G/K} \big( ( \pi(L) + K_n \big) \cap \big(\pi(L) + K_n^c \big) \big) }{\mu_{G/K} ( K_n)}.
\end{align*} 
Since $\pi(L)$ is compact in $G/K$, the strong F\o lner property of $(K_n)_{n \in \mathbb{N}}$ yields that the right-hand side converges to $0$ as $n \to \infty$. Thus, $(\pi^{-1} (K_n))_{n \in \mathbb{N}}$ is a strong F\o lner sequence in $G$. 
\end{proof}

The following result is a reformulation of  \Cref{thm:intro_LCA} in terms of sampling in the Paley-Wiener space.

\begin{theorem} \label{thm:main_lca2}
Let $G$ be a second countable LCA group with Haar measure $\mu_G$. Let $\widehat{G}$ be its dual group with dual measure $\mu_{\widehat{G}}$. 

Let $\ve > 0$. For any compact set $\Omega \subseteq \widehat{G}$, there exists $T \subseteq G$ of uniform Beurling density $D(T) \leq (1+\varepsilon) \mu_{\widehat{G}} (\Omega)$ such that
\begin{align} \label{eq:inequalities3}
 A(\varepsilon) \mu_{\widehat G}(\Omega)  \| f \|^2 \leq \sum_{t \in T} |f(t)|^2 \leq B(\varepsilon) \mu_{\widehat G} (\Omega) \| f \|^2 \qquad \text{for all } f \in \PW_{\Omega} (G),
\end{align}
 where the constants $A(\ve)$ and $B(\ve)$ are given as in \eqref{eq:frameconstants}.
 \end{theorem}
\begin{proof}
We split the proof into two steps.
\\~\\
\textbf{Step 1.} We first prove the result under the additional assumption that $\widehat{G}$ is compactly generated.
Let $Q_H$ be the fundamental domain of the reference lattice $H$ in $G$ and let $\Sigma_{H^{\perp}}$ be the fundamental domain of $H^{\perp}$ in $\widehat{G}$. 

An application of Theorem \ref{thm:main_lca} with the scaled Haar measure $1/ \mu_{\widehat{G}} (\Sigma_{H^{\perp}}) \mu_{\widehat{G}}$  yields a set $T \subseteq G$ of uniform density
\begin{align} \label{eq:densities_normalized}
D_H(T) \leq  (1+\varepsilon) \frac{\mu_{\widehat{G}} (\Omega)}{\mu_{\widehat{G}} (\Sigma_{H^{\perp}})}
\end{align}
 such that
\begin{align*}
 A(\varepsilon) \frac{\mu_{\widehat G}(\Omega)}{\mu_{\widehat{G}} (\Sigma_{H^{\perp}})^2}  \| f \|^2 \leq \frac{1}{\mu_{\widehat{G}} (\Sigma_{H^{\perp}})^2} \sum_{t \in T} |\langle f, e_t \rangle |^2 \leq B(\varepsilon) \frac{\mu_{\widehat G}(\Omega)}{\mu_{\widehat{G}} (\Sigma_{H^{\perp}})^2} \| f \|^2
\end{align*}
for all $f \in L^2 (\Omega; \mu_{\widehat G})$. Here, $A(\varepsilon), B(\varepsilon) > 0$ are the constants given in \eqref{eq:frameconstants}.

Since $\mu_G (Q_H) = 1/\mu_{\widehat{G}} (\Sigma_{H^{\perp}})$ by duality, it follows from Lemma \ref{lem:beurlingleptin}, Lemma \ref{lem:densities} and the inequality \eqref{eq:densities_normalized} that
\[
D(T) = \frac{D_H (T)}{\mu_G (Q_H)} \leq (1+\varepsilon) \mu_{\widehat{G}}(\Omega),
\]
which yields the desired result.
\\~\\
\textbf{Step 2.} 
In this step, we prove the general case. Let $\Gamma \subseteq \widehat{G}$ be the open subgroup generated by $\Omega$. By Lemma \ref{lem:notcompactlygenerated}, $\Gamma$ is compactly generated and $\widehat{G/\Gamma^{\perp}} = \Gamma$. Suppose $\mu_{\Gamma^{\perp}}$ is the normalized measure on $\Gamma^{\perp}$ and let $\mu_{G_0}$ be Haar measure on $G_0 := G/\Gamma^{\perp}$ so that Weil's integral formula \eqref{eq:weil} holds, that is,  $\mu_{G_0} = \mu_G (\pi^{-1} (\cdot))$, where $\pi : G \to G/\Gamma^{\perp}$ is the canonical projection. Then also the associated dual measures $\mu_{\widehat{G}}$, $\mu_{\Gamma}$ and $\mu_{\widehat{G}/\Gamma}$ on $\widehat{G}$, $\Gamma \cong \widehat{G / \Gamma^{\perp}}$ and $\widehat{G}/\Gamma \cong \widehat{\Gamma^{\perp}}$ satisfy Weil's integral formula \eqref{eq:weil}, see, e.g., \cite[Theorem 5.5.12]{reiter2000classical}. Since $\Gamma^{\perp}$ is compact and equipped with normalized measure, the dual measure on $\widehat{G}/\Gamma$ is a counting measure. This implies that the dual measure $\mu_{\Gamma}$ on $\Gamma$  is the restriction of $\mu_{\widehat{G}}$ to $\Gamma$.  In particular, we have $\mu_{\Gamma} (\Omega) = \mu_{\widehat{G}} (\Omega)$. 

By Step 1, there exists a set $T_0 \subseteq G_0$ with uniform Beurling density in $G_0$,
\[
D_0 (T_0) \leq (1+\ve) \mu_{\widehat{G}} (\Omega),
\]
and such that
\begin{align*} 
A(\ve) \mu_{\widehat{G}} (\Omega) \| f \|^2 \leq \sum_{t \in T_0} |\langle f, e_{t} \rangle |^2 \leq B(\ve) \mu_{\widehat{G}} (\Omega) \| f \|^2 \quad \text{for all} \quad f \in L^2 (\Omega; \mu_{\widehat{G}}),
\end{align*} 
where the constants $A(\ve)$ and $B(\ve)$ are given as in \eqref{eq:frameconstants}.
The Fourier inversion formula for $f \in \PW_{\Omega} (G_0)$ gives
\[
f(x) = \int_{\Omega} \widehat{f} (\chi) \chi(x) \; d\mu_{\Gamma} (\chi) = \langle \widehat{f}, e_x \rangle, \quad x \in G_0.
\]
Hence, the above inequalities can be rewritten as
\begin{align} \label{eq:PWG0}
A(\ve) \mu_{\widehat{G}} (\Omega) \| f \|_{L^2 (G_0)}^2 \leq \sum_{t \in T_0} |f(t) |^2 \leq B(\ve) \mu_{\widehat{G}} (\Omega) \| f \|_{L^2 (G_0)}^2 
\end{align}
for all $f \in \PW_{\Omega} (G_0)$.

Let $T \subseteq G$ be a set of representatives for $T_0 \subseteq G/\Gamma^{\perp}$.  Since each $f \in \PW_{\Omega} (G)$ is $\Gamma^{\perp}$-periodic by Lemma \ref{lem:notcompactlygenerated} and $\Omega \subseteq \widehat{G_0}$, it can be treated as a function $f \in \PW_{\Omega} (G_0)$, and
\begin{align*} 
\sum_{t \in T} |f(t)|^2 = \sum_{t \in T_0} |f (t)|^2.
\end{align*}
Note that the $\Gamma^{\perp}$-periodicity implies that this identity is independent of the choice of representatives. Moreover, it follows by Weil's integral formula \eqref{eq:weil}  that $\| f \|^2_{L^2(G)} = \| f \|^2_{L^2 (G_0)}$. Hence, the sampling inequalities \eqref{eq:inequalities3} follow directly from \eqref{eq:PWG0}.

Lastly, we prove that $D(T) = D_0 (T_0)$. For this, let $(K_n)_{n \in \mathbb{N}}$ be a strong F\o lner sequence in $G_0 = G/\Gamma^{\perp}$. Then $(\pi^{-1} (K_n))_{n \in \mathbb{N}}$ is a strong F\o lner sequence in $G$ by Lemma \ref{lem:folnerquotient}. Hence, for $x \in G$,
\begin{align*}
\frac{\# (T_0 \cap ((x + \Gamma^{\perp}) + K_n))}{\mu_{G_0} (K_n)} &= \frac{\# (T_0 \cap  (x +  \pi^{-1} (K_n)))}{\mu_{G} (\pi^{-1} (K_n))} 
= \frac{\# (T \cap  (x +  \pi^{-1} (K_n)))}{\mu_{G} (\pi^{-1} (K_n))},
\end{align*}
which easily yields that $D(T) = D_0 (T_0)$. 
\end{proof}

\section*{Acknowledgements}
The first author was partially supported by NSF grant DMS-2349756.
For J.~v.~V., this research was funded in whole or in part by the Austrian
Science Fund (FWF): 10.55776/J4555 and 10.55776/PAT2545623. 

\bibliographystyle{abbrv}
\bibliography{bib}

\end{document}